\documentclass[a4paper, 10pt]{amsart}
\usepackage{etex}
\usepackage{a4wide}
\usepackage[utf8]{inputenc}
\usepackage[T1]{fontenc}
\usepackage{lmodern}
\usepackage{setspace}
\usepackage{layout}
\usepackage{setspace}
\usepackage[pdftex]{graphicx}
\usepackage{hyphsubst}%
\usepackage[english]{babel}
\usepackage{textcomp}
\usepackage{enumitem}
\usepackage{microtype}
\usepackage{amsthm}
\usepackage{amsmath}
\numberwithin{equation}{section}
\usepackage{amssymb}
\usepackage{mathrsfs}
\usepackage{mathtools}
\usepackage{MnSymbol} 
\usepackage[OT2,OT1]{fontenc}
\usepackage{pdfpages}
\usepackage{lipsum}
\newcommand\cyr
{
\renewcommand\rmdefault{wncyr}
\renewcommand\sfdefault{wncyss}
\renewcommand\encodingdefault{OT2}
\normalfont
\selectfont
}
\DeclareTextFontCommand{\textcyr}{\cyr}

\theoremstyle{plain}
\newtheorem{Lemma}{Lemma}
\newtheorem*{Theorem*}{Theorem}

\theoremstyle{definition}
\newtheorem{Definition}{Definition}
\theoremstyle{remark}

\newcommand{\Q}{\mathbb Q}
\newcommand{\N}{\mathbb N}
\newcommand{\F}{\mathbb F}
\newcommand{\Z}{\mathbb Z}

\newcommand{\mx}{\textup{max}}

\begin{document}

	\title{On Artin's Conjecture: Pairs of Additive Forms}
	\author{Miriam Sophie Kaesberg}
	
	\begin{abstract}
		It is established that for every pair of additive forms $f=\sum_{i=1}^s a_i x_i^k, g=\sum_{i=1}^s b_i x_i^k$ of degree $k$ in $s>2k^2$ variables the equations $f=g=0$ have a non-trivial $p$-adic solution for all odd primes.
	\end{abstract}
	\maketitle

\section{Introduction}
\let\thefootnote\relax\footnotetext{This result is part of the author's PhD thesis submitted to the Georg-August Universit\"at G\"ottingen on 6 November 2020.}
Let $k \ge 1$ be a natural number and $a_i$  and $b_i$ integer coefficient for $1 \le i \le s$. A special case of Artin's conjecture~\cite{artin} states that the pair of additive equations
\begin{align*}
	\sum_{i =1}^s a_i x_i^k= \sum_{i=1}^s b_i x_i^k = 0
\end{align*}
have a non-trivial $p$-adic solution for all primes $p$ provided that $s > 2k^2$.

Davenport and Lewis~\cite{twoaddeq} started to answer the question whether this statement is true by proving that $s > 2k^2$ variables are sufficient if $k$ is odd, whereas for even $k$ they only obtained the bound $s \ge 7k^3$.
Br\"udern and Godinho~\cite{BundG2} proved that the expected bound $s > 2k^2$ holds for even $k$ which are not of the shape
\begin{align*}
    k=3\cdot2^\tau \quad \textup{or} \quad k=p^\tau\left(p-1\right)
\end{align*}
for $p$ prime and $\tau \ge1$ as well. For each of these excluded shapes they proved for all but one prime that a non-trivial $p$-adic solution exists if $s>2k^2$. The missing primes are $p=2$ in the case $k=3 \cdot 2^\tau$ and $p$ if $k=p^{\tau} \left(p-1\right)$. Here, they gave the bounds $s\ge \frac{8}{3} k^2$ for $p=2$ and $k=3 \cdot 2^\tau$, $s \ge 8k^2$ for $p=2$ and $k=2^\tau$, and $s\ge4k^2$ for $p\ge 3$ and $k=p^\tau \left( p-1 \right)$. All in all, the bound $s\ge8k^2$ holds for all $p$ and all $k$.

There was some further progress for $p=2$ and $k=2^\tau$ for $\tau=1$, $\tau=2$ and $\tau\ge 16$. For $k=2$ the expected bound $s > 8$ follows from the general result by Dem'yanov~\cite{Demyanov} that for two quadratic forms $f_1, f_2$ in at least nine variables the equations $f_1=f_2=0$ have a non-trivial $p$-adic solution for all primes $p$. Poehler~\cite{Poehler} proved for $k=4$, that $49=3k^2+1$ variables suffice and Kr\"anzlein~\cite{Kraenzlein} showed for $k=2^\tau$ with $\tau \ge 16$ that the expected $2k^2+1$ variables are sufficient.

For $p \ge 3$ and $k=p^\tau\left(p-1\right)$ on the other hand, the bound was further sharpened by Godinho and de Souza Neto~\cite{GundNet,GodSouNet} who proved that
$
        s \ge 2 \frac{p}{p-1} k^2 -2k
$
suffices for $p \in \{3,5\}$ and if $\tau \ge \frac{p-1}{2}$ for $p \ge7$ as well. For $k=6=3\cdot 2$, the bound $s > 2k^2$ was reached by Godinho, Knapp and Rodrigues~\cite{k=6} while later Godinho and Ventura~\cite{p=3} showed that this bound suffices for $k=3^\tau \cdot 2$ with $\tau \ge 2$ as well. Therefore, all pairs of diagonal forms of equal degree $k$ in more than $2k^2$ variables have a non-trivial $3$-adic solution. The aim of this paper is to prove the following theorem, which shows that this statement does not only hold for $p=3$ but for all $p\ge 3$, by taking care of the degrees $k=p^{\tau} \left(p-1\right)$ for $p \ge 5$ and $\tau \ge 1$.
\begin{Theorem*}
        Let $p \ge 5$ be a prime, $\tau \ge1$ and $k=p^\tau\left(p-1\right)$. Then for $a_i, b_i \in \Z$ with $1 \le i\le s$, the equations 
        \begin{align}
        \label{Theoremeq}
            \sum_{i =1}^s a_i x_i^k= \sum_{i=1}^s b_i x_i^k = 0
        \end{align}
        have a non-trivial $p$-adic solution for all $s >2k^2$.
\end{Theorem*}
This completes the proof of Artin's conjecture for two diagonal forms of the same degree for all primes $p\neq 2$. For $p=2$ there are only the questions left whether there is a non-trivial $2$-adic solution for $k=3 \cdot 2^\tau$ for $\tau \ge 2$ and $k=2^\tau$ for $2 \le \tau \le 15$ provided that $s > 2k^2$. The argument by Kr\"anzlein~\cite{Kraenzlein} can be easily applied for the case $k=3\cdot 2^\tau$ as well if $\tau\ge 16$. Thus, only finitely many $k$ remain for which the bound $s > 2k^2$ is not reached.
\newline\\
The proof of the theorem follows a pattern by Davenport and Lewis~\cite{twoaddeq} while making use of some improvements by Br\"udern and Godinho~\cite{BundG2}. Section \ref{$p$-normalization} defines an equivalence relation on the set of all systems \eqref{Theoremeq}, introduced by Davenport and Lewis~\cite{twoaddeq}. This equivalence relation is defined in a way that solubility of \eqref{Theoremeq} in $\Q_p^s\backslash\{\boldsymbol{0}\}$ is preserved, which allows to pick representatives with useful properties from each class and prove the existence of a non-trivial $p$-adic solution only for them. Due to a version of Hensel's lemma, one can show that a system \eqref{Theoremeq} has a non-trivial $p$-adic solution by proving that the congruences
\begin{align}
\label{equiv}
    \sum_{i = 1}^s a_i x_i^k \equiv \sum_{i=1}^s b_i x_i^k \equiv 0 \mod p^{\tau+1}
\end{align}
have a solution for which the matrix
\begin{align}
\label{matrix}
		\begin{pmatrix}
		a_1 x_1 & \dots & a_s x_s\\
		b_1 x_1 & \dots & b_s x_s
		\end{pmatrix}
		\end{align}
has rank $2$ modulo $p$. Section \ref{Coloured variables and Contractions} recalls the notions of coloured variables, introduced by Br\"udern and Godinho~\cite{BundG2}, and contractions which were established by Davenport and Lewis~\cite{twoaddeq}. Together, they are the foundation of the proof. Coloured variables and a refinement of them provide a way to take care of the rank of the matrix \eqref{matrix}, while contractions are a means to solve the equations \eqref{equiv} recursively by lifting solutions modulo $p^l$ to solutions modulo $p^{l+1}$. Furthermore, this section continues the path laid down by Davenport and Lewis~\cite{twoaddeq} and Br\"udern and Godinho~\cite{BundG2}, which issues more restrictions on the pairs of equations one has to find a solution for. Section \ref{Combinatorial Results} is a collection of combinatorial results which are frequently used, directly and indirectly, in the remaining sections. A description on how the notion of coloured variables is used in combination with contractions to obtain a solution of \eqref{equiv} such that the matrix \eqref{matrix} has rank $2$ is contained in Section \ref{Strategy}, whereas Section \ref{Contraction Related Auxiliaries} consists of a collection of lemmata which describe situations in which one can lift some solutions modulo $p^l$ to solutions of a higher modulus. The remaining two sections contain the actual proof which is divided into Section \ref{Pairs of forms with tau=1} for the case $k=p\left(p-1\right)$ and Section \ref{Pairs of forms with tau ge 2}, where the remaining cases with $k= p^\tau \left(p-1\right)$ and $\tau \ge 2$ are handled. This division is due to the different modulus in \eqref{equiv}. For big $\tau$, one has more variables whose coefficients are not both congruent to $0$ modulo $p^{\tau+1}$, which is balanced in the case $\tau=1$ by a permutation argument.
\section{$p$-Normalisation}
	\label{$p$-normalization}
    This section will recall an equivalence relation on the set of systems \eqref{Theoremeq} which was introduced by Davenport and Lewis~\cite{twoaddeq} in order to choose representatives with specific characteristics.
    
    Define for any pair of additive forms
    \begin{align}
    \label{f,g}
        f= \sum_{i=1}^s a_i x_i^k,\qquad g= \sum_{i=1}^s b_i x_i^k
    \end{align}
    with rational coefficients $a_i$ and $b_i$ $\left( 1 \le i \le s \right)$ a rational number
    \begin{align*}  
        \vartheta \left( f, g \right) := \prod_{\substack{1 \le i, j \le s\\ i \neq j }} \left( a_i b_j - a_j b_i \right).
    \end{align*}
    For integers $\nu_i$ $\left( 1 \le i \le s \right)$ consider the pair
    \begin{align}
    \label{op1}
        f'= f\left(p^{\nu_1}x_1, \dots , p^{\nu_s}x_s\right), \qquad g' = g\left(p^{\nu_1}x_1, \dots , p^{\nu_s}x_s\right)
    \end{align}
    and for rational numbers $\lambda_1$, $\lambda_2$, $\mu_1$ and $\mu_2$ with $\lambda_1 \mu_2 - \lambda_2 \mu_1 \neq 0$ the pair
    \begin{align}
    \label{op2}
        f''=\lambda_1 f + \lambda_2 g, \qquad g''= \mu_1 f + \mu_2 g. 
    \end{align}
    If another pair $\tilde f, \tilde g$ with rational coefficients can be obtained by a finite succession of the operations $\eqref{op1}$ and $\eqref{op2}$ on the pair $f, g$, then they are called $p$-equivalent. If $\left( x_1', \dots, x_s'\right)$ is a non-trivial solution of $f'=g'=0$ then $\left(p^{\nu_1}x_1', \dots, p^{\nu_s}x_s'\right)$ is a non-trivial solution of $f=g=0$, whereas if $\left( x_1, \dots, x_s\right) $ is a non-trivial solution for $f=g=0$, then one has a non-trivial solution for $f'=g'=0$ as well, given via $\left( p^{-\nu_1} x_1, \dots, p^{-\nu_s}x_s\right)$. Therefore, solubility is preserved under the operation $\eqref{op1}$. The same holds for the operation $\eqref{op2}$. Here, one direction is obvious, and the other holds, because the transformation is invertible. Consequently, the existence of a non-trivial solution for $f=g=0$ in $\Q_p$ implies that there is one for all pairs $\tilde f, \tilde g$ which are $p$-equivalent to $f, g$. It can also be easily deduced from the definition of $\vartheta \left( f, g\right)$, that if $\vartheta \left(f, g\right)=0$, the same holds for $\vartheta \left(f',g'\right)$ and $\vartheta \left( f'', g''\right)$ and  therefore, for the whole $p$-equivalent class.
    \begin{Definition}
        A pair $f,g$ given by \eqref{f,g} with integers coefficients and $\vartheta \left( f, g \right) \neq 0 $ is called \textit{$p$-normalized}, if the power of $p$ dividing $\vartheta \left( f, g\right)$ is as small as possible amongst all pairs of forms \eqref{f,g} with integer coefficients in the same $p$-equivalent class.
    \end{Definition}
	As each $p$-equivalent class contains pairs, for which all coefficients $a_i, b_i$ are integers, it follows that the existence of a non-trivial solution for all $p$-normalised pairs induces a non-trivial solution for all pairs of forms with rational coefficients $a_j, b_j$ and $\vartheta \left( f, g \right)\neq 0$. Using a compactness argument, Davenport and Lewis~\cite{twoaddeq} showed that it induces the existence of a solution for all pairs of forms $f, g$ with $\vartheta\left( f, g \right)=0$ as well.
    \begin{Lemma}
    \label{pnormreicht}
        Suppose for an fixed $s$ that the equations $f=g=0$ have a non-trivial solution in $\Q_p$ for all $p$-normalised pairs $f,g$. Then, for any rational coefficients $a_j, b_j$, the equations \eqref{Theoremeq} have a non-trivial solution in $\Q_p$.
    \end{Lemma}
    \begin{proof}
        See~\cite[Section 5]{twoaddeq}.
    \end{proof}
    Consequently, it suffices to focus on finding non-trivial $p$-adic solutions for $p$-normalised pairs $f, g$ in more than $2k^2$ variables. The following lemma gives information about the properties of them.
     \begin{Lemma}
    	\label{order}
    	A $p$-normalised pair of additive forms $f,g$ of degree $k$ in $s$ variables can be written as
    	\begin{align*}
    		f&=f_0 + p f_1 + \dots + p^{k-1} f_{k-1},\\
    		g&=g_0+ p g_1 + \dots + p^{k-1} g_{k-1},
    	\end{align*}
    	where $f_i, g_i$ are forms in $m_i$ variables, and these sets of variables are disjoint for $i = 0, 1, \dots, k-1$. Moreover, each of the $m_i$ variables occurs in at least one of $f_i, g_i$ with a coefficient not divisible by $p$. One has 
    	\begin{align}
    		m_0 + \dots + m_j \ge \frac{\left(j + 1\right) s}{k} \qquad \textup{ for } \qquad j=0, 1, \dots, k-1.
    	\end{align}
    	Moreover, if $q_i$ denotes the minimum number of variables appearing in any form $\lambda f_i + \mu g_i$ ($\lambda$ and $ \mu$ not both divisible by $p$) with coefficients not divisible by $p$, then 
    	\begin{align*}	
    			m_0 + \dots + m_{j-1} + q_j \ge \frac{\left( j + \frac 1 2 \right) s }{k } \qquad \textup{ for } \qquad j = 0, 1, \dots, k-1.
    	\end{align*}
    \end{Lemma}
	\begin{proof}
		See~\cite[Lemma 9]{twoaddeq}.
	\end{proof}
	At least one integer coefficient $a_i$ or $b_i$ of a variable $x_i$ of a $p$-normalised pair $f, g$ is non-zero, because else one would have $\vartheta\left( f, g\right)=0$. Consequently, there is a maximal power $l$ of $p$, which divides both $a_i$ and $b_i$. Due to the previous lemma, one can deduce, that $0 \le l\le k-1$ for all variables $x_i$ of a $p$-normalised pair.
	\begin{Definition}
        A variable $x_i$ of a pair $f, g$ with integer coefficients is said to be \textit{at level} $l$ if its coefficients $a_i$ and $b_i$ are both divisible by $p^l$ but not both divisible by $p^{l+1}$.
	\end{Definition}
	By Lemma \ref{order}, a $p$-normalised pair has exactly $m_l$ variables at level $l$ for $0 \le l \le k-1$. The integers $\tilde a_i, \tilde b_i$ are defined for a variable $x_i$ at level $l$ with integer coefficients $a_i, b_i$ via $\tilde a_i= p^{-l}a_i$ and $\tilde b_i= p^{-l} b_i$. These integers $\tilde a_i, \tilde b_i$ are the coefficients of the forms $f_l, g_l$ as defined in Lemma~\ref{order} and the vector $\binom{\tilde a_i}{\tilde b_i}$ is called the \textit{level coefficient vector} of a variable $x_i$.
    
    One can restrict the question of the existence of a non-trivial $p$-adic solution to one of congruences. To this end, it is useful to adopt the notation $k=p^\tau \delta k_0$ with $\delta = \textup{gcd}\left(k, p-1 \right) $, $\textup{gcd}\left(p,k_0 \right) =1$ and
	\begin{align}
	\label{gamma}
	\gamma := \begin{cases}
	1, & \textup{ if } \tau = 0\\
	\tau +1, & \textup{ if } \tau > 0  \textup{ and } p >2\\
	\tau +2, & \textup{ if } \tau > 0 \textup{ and } p=2,
	\end{cases}
	\end{align}
	by Davenport and Lewis~\cite{twoaddeq} which is used in the following lemma.
	\begin{Lemma}
		\label{Hensel}
		If the congruences
		\begin{align}
			\label{modeq}
		\sum_{i = 1}^s a_i x_i^k \equiv 0 \mod p^\gamma, \qquad \sum_{i=1}^s b_i x_i^k \equiv 0 \mod p^\gamma
		\end{align}
		have a solution in the integers for which the matrix
		\begin{align*}
		\begin{pmatrix}
		a_1 x_1 & \dots & a_s x_s\\
		b_1 x_1 & \dots & b_s x_s
		\end{pmatrix}
		\end{align*}
		has rank $2$ modulo $p$, then the equations \eqref{Theoremeq} have a non-trivial $p$-adic solution.
	\end{Lemma}
	\begin{proof}
		See~\cite[Lemma 7]{twoaddeq}.
	\end{proof}
	Such a solution is called a \textit{non-singular solution}. The remainder of the proof of the theorem will focus on finding non-singular solutions for $p$-normalised pairs $f, g$.
	
	The next section will introduce the methods used to find non-singular solutions.
	\section{Coloured Variables and Contractions}
	\label{Coloured variables and Contractions}
	This section will recall the concept of coloured variables, first used by Br\"udern and Godinho~\cite{BundG2}, and refine it in a way such that it meets the requirements of the special case $k=p^{\tau} \left(p-1\right)$. It will also describe the method of contractions which was introduced by Davenport and Lewis~\cite{twoaddeq}. Together, both concepts form the foundation of this proof.
	
    To have more control over the non-singularity of a solution of \eqref{modeq}, Br\"udern and Godinho~\cite{BundG2} divided the set of variables at level $l$ into $p+1$ sets, depending on their level coefficient vector.
	For that, they defined the vectors $\boldsymbol{e}_0 = \binom1 0$ and $\boldsymbol{e}_\nu = \binom \nu 1 $ for $\nu \in \left \{1, \dots, p \right \}$. Viewed as vectors in $\left(\Z/p\Z\right)^2$ the vectors define the sets
	\begin{align*}                                                                                                                                                                                     
            \mathscr{L}_\nu :=\left \{ c \boldsymbol{e}_\nu \mid c \in \left(\Z/p\Z\right)^\ast  \right \}
    \end{align*}
        for $0 \le \nu \le p$. Modulo $p$, each level coefficient vector $\left( \tilde{a}_i, \tilde{b}_i \right)$ lies in exactly one of the disjoint sets $\mathscr{L}_\nu$.
    \begin{Definition}
        A variable $x_i$ at level $l$ is said to be of \textit{colour} $\nu$, if the level coefficient vector $\left(\tilde a_i, \tilde b_i\right)$ interpreted as a vector in $\F_p^2$ lies in $\mathscr{L}_\nu$. The parameter $I_\nu^l$ of a pair $f, g$ is the number of variables $x_i$ at level $l$ of colour $\nu$.
    \end{Definition}
	The parameter $q_l$ introduced in Lemma \ref{order} denotes the minimum number of variables appearing with a coefficient not divisible by $p$ in any form $\lambda f_l + \mu g_l$ with $\left( \lambda, \mu \right) \nequiv \left(0,0\right)$ modulo $p$. This is closely related to the concept of coloured variables. By setting $\lambda \equiv 0$ modulo $p$ for $\nu=0$ or $\mu\equiv -\lambda \nu$ for $\nu \in \{1, \dots, p\}$ the variables which appear in $\lambda f_l + \mu g_l $ with a coefficient divisible by $p$ are exactly those of colour $\nu$. Consequently, if $I_\nu^l\ge I_\mu^l$ for all $0 \le \mu \le p$ it follows that $I_\nu^l= m_l-q_l$. Define  $I_\mx^l=m_l-q_l$. This notation can be generalised as follows.
	\begin{Definition}
		For a set $\mathscr{K}$ of indices $i$ of variables $x_i$ at level $l$ define $I_\nu\left(\mathscr{K}\right)$ as the number of $i \in \mathscr{K}$ with $x_i$ of colour $\nu$, $I_\mx\left(\mathscr{K}\right)= \max_{0 \le \nu \le p} I_\nu\left(\mathscr{K}\right)$ and $q\left( \mathscr{K} \right)= |\mathscr{K} |- I_\mx\left( \mathscr{K}\right)$.
	\end{Definition}
	Note, that if $\mathscr{K}$ is the set of all indices of variables at level $l$, then $|\mathscr{K}|=m_l$, $I_\nu\left(\mathscr{K}\right)=I_\nu^l$, $I_\mx\left(\mathscr{K}\right)=I_\mx^l$ and $q\left( \mathscr{K} \right)=q_l$.
	
	From the definition of a non-singular solution it follows, that whether a solution of $\eqref{modeq}$ is non-singular depends exclusively on the variables at level $0$. If a solution of \eqref{modeq} has variables at level $0$ of at least two different colours set to a value which is not congruent to $0$ modulo $p$, the corresponding matrix has rank $2$ modulo $p$ making it a non-singular solution.
	To use variables at different levels one can take sets of variables at one level and combine them in a way that they can be seen as a variable of a higher level. This method was introduced by Davenport and Lewis~\cite{twoaddeq} and applied in combination with the notion of coloured variables by Br\"udern and Godinho~\cite{BundG2}.
	\begin{Definition}
		Let $\mathscr{K}$ be a set of indices $j$ with $x_j$ at level $l$. Let $h \in \N$ with $h>l$ and suppose that there are integers $y_j$ with $p \nmid y_j$ such that
		\begin{align}
		\label{contraction}
		\sum_{j \in \mathscr{K}} a_j y_j^k \equiv \sum_{j \in \mathscr{K}} b_j y_j^k \equiv 0 \mod p^{h}.
		\end{align}
		Then $\mathscr{K}$ is called a \textit{contraction from level $l$ to level at least $h$}. If either $\sum_{j \in \mathscr{K}} a_j y_j^k$ or $\sum_{j \in \mathscr{K}} b_j y_j^k$ is not congruent to $0$ modulo $p^{h+1}$, then $\mathscr{K}$ is called a \textit{contraction from level $l$ to level $h$}.
	\end{Definition}
	Recall for variables at level $l$ that $\tilde a_j = p^{-l} a_j$ and $\tilde b_j = p^{-l} b_j$. Hence, a set $\mathscr{K}$ of variables at level $l$ is a contraction to a variable at level at least $l+n$ if there are $y_j$ not divisible by $p$ such that
	\begin{align*}
        \sum_{j \in \mathscr{K}} \tilde a_j y_j^k \equiv \sum_{j \in \mathscr{K}} \tilde b_j y_j^k \equiv 0 \mod p^n.
	\end{align*}

	If $\mathscr{K}$ is a contraction from level $l$ to some level $h$, one can set $x_j=y_jX_0$ for all $j$ in the contraction $\mathscr{K}$. Through this, one obtains a variable $X_0$ at level $h$. One says that the variable $X_0$ can be \textit{traced back} to the variables $x_j$ with $j \in \mathscr{K}$. Assume that there are other variables $X_i$ at level $h$ with $i \in \{1, \dots, n\}$, where each of the variables $X_i$ is a variable at level $h$ which either occurred in the pair $f, g$ or is the result of a contraction. If the set of indices $\{0, 1, \dots, n\}$ of the variables $X_0, X_1, \dots, X_n$ is a contraction to a variable $Y$ at a level at least $h+1$, then one says that the variable $Y$ can be traced back not only to the variables $X_i$ for $i \in \{ 0, 1, \dots, n\}$ but also to all the variables that those variables can be traced back to. For example, $Y$ can be traced back to all $x_j$ with $j \in \mathscr{K}$.
	\begin{Definition}
	A variable is called a \textit{primary variable} if it can be traced back to two variables at level $0$ of different colours.
	\end{Definition}
	If one can contract a primary variable at level at least $\gamma$, then by setting this contracted variable $1$ and everything else zero, one obtains a non-singular solution of \eqref{modeq} and therefore a non-trivial $p$-adic solution.
	
	In some cases the knowledge of the exact level and colour of a variable that was contracted will give quite an advantage. To gain control about this, the concept of coloured variables is not strong enough because it can only give the information whether a certain set of variables at level $l$ is a contraction to a variable at level at least $l+1$, but one does not know the behaviour of the variables modulo $p^{l+2}$. Therefore, one cannot use it to extract information about the exact level and colour of the contracted variable. To gain this information, one can divide the set of variables of one colour into smaller sets, which consider the level coefficient vectors $\binom{\tilde{a}_i}{\tilde{b}_i}$ not only modulo $p$ but modulo $p^2$.
	
	For that, view the vectors $\boldsymbol{e}_0 = \binom1 0$ and $\boldsymbol{e}_\nu = \binom \nu 1$ as vectors in $\left(\Z/p^2\Z\right)^2$ and define the vectors $\boldsymbol e^0= \binom 0 p $ and $\boldsymbol e^\nu = \binom p 0 $ for $\nu \in \{1, \dots, p-1\}$. This enables one to define sets similar to the sets $\mathscr{L}_\nu$ via
        \begin{align*}
            \mathscr{L}_{\nu \mu} &:= \left \{ c\left(\boldsymbol e_\nu +\mu \boldsymbol e^\nu \right) \mid c \in \left(\Z/p^2\Z\right)^\ast \right\}
        \end{align*}
        for $0 \le \nu \le p$ and $0 \le \mu \le p-1$.
	Here again, a level coefficient vector $\binom{\tilde{a}_i}{\tilde{b}_i}$ lies modulo $p^2$ in exactly one of the disjoint sets $\mathscr{L}_{\nu \mu}$.
	\begin{Definition}
		A variable $x_i$ is said to be of \textit{colour nuance} $\left(\nu,\mu\right)$ if the level coefficient vector $\left(\tilde a_i, \tilde b_i\right)$ interpreted as a vector in $\left(\Z/p^2 \Z\right)^2$  lies in $\mathscr{L}_{\nu \mu}$. The parameter $I_{\nu \mu}^l$ of a pair $f, g$ is the number of variables $x_i$ at level $l$ of colour nuance  $\left(\nu,\mu\right)$.
	\end{Definition}
	For all variables $x_i$ of colour nuance $\left(\nu, \mu\right)$ there is a unique integer $c_i \in \{1, 2, \dots, p^2\}\backslash p\Z$ for which $\binom{\tilde{a}_i}{\tilde{b}_i} \equiv c_i \left( \boldsymbol{e}_\nu + \mu \boldsymbol{e}^\nu\right) \bmod p^2$. The integer $c_i$ is said to be the \textit{corresponding integer} to $x_i$.
	
	Lemmata \ref{pnormreicht} and \ref{Hensel} show that it suffices to find a non-singular solution for all $p$-normalised pairs in order to prove that for any rational coefficients $a_j, b_j$ the equations \eqref{Theoremeq} have a non-trivial solution in $\Q_p$. Due to Lemma \ref{order} one already has some information about the number of variables at certain levels and the distribution of these variables in the different colours of $p$-normalised forms $f, g$. One can further exploit that every $p$-equivalence class contains more than just one $p$-normalised pair. The next lemma shows further properties that are fulfilled by at least one $p$-normalised pair in each $p$-equivalence class for which $\vartheta\left(f, g\right) \neq 0$ holds.
	
	\begin{Lemma}
		\label{properties}
		Each pair of additive forms \eqref{f,g}, with rational coefficients and $\vartheta \neq 0$, is $p$-equivalent to a $p$-normalised pair $f, g$ possessing the following properties:
		\begin{enumerate}
			\item $g_0$ contains exactly $q_0$ variables with coefficients not divisible by $p$.
			\item One of $f_1, g_1$ contains exactly $q_1$ variables with coefficients not divisible by $p$.
			\item $g_0$ has the form 
				\begin{align*}	
					g_0 = p^2 \sum_{i=1}^{I_{00}^0} \alpha_i x_i^k + p \sum_{i=I_{00}^0+1}^{I_0^0} \beta_i x_i^k + \sum_{I_0^0 +1}^{m_0} \gamma_i x_i^k,
				\end{align*}
				where $\beta_{I_{00}^0+1}, \dots, \beta_{I_0^0}, \gamma_{I_0^0+1}, \dots, \gamma_{m_0}$ are not divisible by $p$, and 
			\begin{align*}
			m_0 + m_1-I_0^1 - \frac{s}{k} \ge I_{00}^0 \ge \frac{m_0-q_0}{p}.
			\end{align*}
		Furthermore, $I_{00}^{0} \ge I_{0 \mu}^0$ for all $0 \le \mu \le p-1$.
		\end{enumerate}
	\end{Lemma}
	\begin{proof}
		See~\cite[Lemma 10]{twoaddeq}.
	\end{proof}
	It follows from the first property, that $I_\mx^0=I_0^0=m_0-q_0$. The second property shows, that either $I_0^1 = m_1-q_1$ or $I_p^1=m_1-q_1$ and therefore, either the colour $0$ or the colour $p$ has the most variables at level $1$. Note, that it follows from the third property, that 
	\begin{align*}
		 I_0^0 + q_0 + m_1- I_0^1 -\frac{s}{k} \ge I_{00}^0 \ge \frac{I_0^0}{p}
	\end{align*}
	and thus, that
	\begin{align}
	\label{I_0^0ohneI_{00}^0}
		I_0^0 - I_{00}^0 \ge \frac{s}{k} -q_0 - \left( m_1-I_0^1\right).
	\end{align}	
	As every $p$-normalised pair is $p$-equivalent to a $p$-normalised pair possessing the properties of the previous lemma, it suffices to prove the existence of a non-singular solutions for $p$-normalised pairs with these properties.
	
	By using only the variables at level $0$ it was proved by Br\"udern and Godinho~\cite[Section 4]{BundG2} that a pair $f, g$ for which $q_0$ is large has a non-singular solution as displayed in the following.
	
	They said that a colour $\nu$ is \textit{zero-representing} if there is a subset $\mathscr{K}$ of variables at level $0$ of colour $\nu$ for some $0\le \nu\le p$, which is a contraction to a variable at level at least $\gamma$.
    The following Lemma is an immediate result from this definition.
	\begin{Lemma}
		\label{twocolourszerorepresenting}
		If a pair $f, g$ has two colours that are zero-representing, then there exists a non-singular solution of \eqref{modeq}.
	\end{Lemma}
	\begin{proof}
        See~\cite[Lemma 4.1]{BundG2}.
	\end{proof}
   Using a theorem of Olson~\cite{Olson1}, they then provided a lower bound of the amount of variables at level $0$ of colour $\nu$ which are required in order to ensure that $\nu$ is zero-representing.
	\begin{Lemma}
	\label{zeropresenting}
		If $I_{\nu}^0\ge p^{\gamma} + p^{\gamma  -1} -1$, then the colour $\nu$ is zero-representing.
	\end{Lemma}
	\begin{proof}
		See~\cite[Lemma 4.2]{BundG2}.
	\end{proof}
	Using these two lemmata and the theorem of Olson~\cite{Olson1} again, they concluded the following statement.
	\begin{Lemma}
	\label{q_0}
		If a pair $f, g$ has $q_0\ge 2p^{\gamma}-1$, then there exists a non-singular solution of \eqref{modeq}.
	\end{Lemma}
	\begin{proof}
		See~\cite[Lemma 4.4]{BundG2}
	\end{proof}
	Therefore, it suffices to focus on $p$-normalised forms $f, g$ that fulfil the properties of Lemma~\ref{properties} and have $q_0 \le 2p^{\gamma} -2$.
	\section{Combinatorial Results}
	\label{Combinatorial Results}
	This section contains a collection of lemmata with combinatorial results on congruences modulo~$p$ and $p^2$ for primes $p$, which will later be convenient for finding contraction in certain sets.
	\begin{Lemma}
		\label{nicht0}
		Let $n > \textup{ggT}\left( k, p-1 \right)$ and $c_1, \dots , c_n$ be any integers coprime to $p$. Then, the congruence 
		\begin{align*}
			c_1x_1^k + \dots + c_n x_n^k \equiv 0 \mod p
		\end{align*}
		has a solution with $x_1 \nequiv 0 \bmod p$.
	\end{Lemma}
	\begin{proof}
		See~\cite[Lemma 1]{homaddeq}.
	\end{proof}
	\begin{Lemma}
        \label{p^n}
        Let $\alpha_{ij} \in \Z$ for $1 \le i \le n$ and  $ 1\le j \le s$  with $s \ge np-n+1$. Then the equation
        \begin{align*}
            \sum_{j=1}^s \varepsilon_j \begin{pmatrix}
                                        \alpha_{1j}\\ \vdots \\ \alpha_{nj} 
                                       \end{pmatrix} \equiv 0 \mod p
        \end{align*}
        has a solution with $\varepsilon_j \in \left\{0,1\right\}$ for $ 1\le j \le s$ and some $\varepsilon_j \neq 0$.
	\end{Lemma}
    \begin{proof}
        This is the special case $G = \left(\Z/p\Z\right)^n $ of the theorem of Olson~\cite{Olson1}.
    \end{proof}
    \begin{Lemma}
        \label{FS1C}
        Let $ s \ge 3p-2$ and $a_j, b_j \in \Z$ for $1 \le j \le s$. Then there exists a non-empty subset $J \subset \{1, 2, \dots, s \}$ with $|J| \le p$ and $\sum_{j \in J} a_j \equiv \sum_{j \in J} b_j \equiv 0 \bmod p$.
	\end{Lemma}
    \begin{proof}
        See~\cite[Lemma 1.1]{Olson2}.
	\end{proof}
    \begin{Lemma}
        \label{FS4}
        Let $d_j \in \Z\backslash p \Z$ for $ 1\le j \le 3p-2$. Then there exists a non-empty subset $J \subset \left \{ 1, \cdots, 3p-2\right \}$ with $|J| \le p$,
        \begin{align*}
            \sum_{j \in J} d_j \equiv 0 \mod p \qquad \textup{ and } \qquad \sum_{j \in J} d_j \nequiv 0 \mod p^2.
        \end{align*}
    \end{Lemma}
    \begin{proof}
        See~\cite[Lemma 3.7]{GodSouNet}.
    \end{proof}
	\begin{Lemma}
		\label{2p-1fuerp=5}
		Let $d_j \in \Z\backslash 5 \Z$ for $ 1\le j \le 9$. Then there exists a non-empty subset $J \subset \left \{ 1, \cdots, 9\right \}$ with $|J| \le 5$,
		\begin{align*}
		\sum_{j \in J} d_j \equiv 0 \mod 5 \qquad \textup{ and } \qquad \sum_{j \in J} d_j \nequiv 0 \mod 25.
		\end{align*}
	\end{Lemma}
	\begin{proof}
		See~\cite[Proposition 3.1]{GundNet}
	\end{proof}
    \section{Strategy}
    \label{Strategy}
        This section contains a general description of the remainder of the proof, for which further notation is introduced. Assume for the remainder of this paper that $\tau\ge1$ is an integer, $p\ge 5$ a prime and $k=p^{\tau}\left(p-1\right)$. This will not be repeated in the following but nonetheless assumed in all following lemmata.
        \begin{Definition}
        	A $p$-normalised pair of additive forms $f, g$ as in \eqref{f,g} is called a \textit{proper $p$-normalised pair} if $s\ge 2k^2+1$, $q_0\le 2p^{\tau+1}-2$ and it satisfies the properties of Lemma \ref{properties}. 
        \end{Definition}
        The restrictions on $k$, $p$ and $\tau$ show that $\gamma=\tau+1$. Therefore, it follows from Lemmata \ref{pnormreicht}, \ref{properties} and \ref{q_0}, that it suffices to prove for every proper $p$-normalised pair $f, g$ that the equations $f=g=0$ have a non-trivial $p$-adic solution.
        
        The bound $s \ge 2k^2+1$ and Lemma \ref{order} show, that a proper $p$-normalised pair has the lower bounds
        \begin{align*}
            m_0 + \dots + m_j &\ge \left(2j+2\right)p^{\tau+1}-\left(2j+2\right) p^{\tau}+1,\\
            m_0 + \dots + m_{j-1} + q_j &\ge \left(2j+1\right) p^{\tau+1} -\left(2j+1 \right) p^\tau +1
        \end{align*}
        for $j \in \{0, \dots, k-1\}$ and Lemma \ref{properties} provides furthermore
        \begin{align}
        \label{NichtI_00Grenze}
            I_0^0 - I_{00}^0 \ge 2p^{\tau+1} -2p^{\tau} -q_0-\left(m_1-I_0^1\right).
        \end{align}
        To find a non-trivial $p$-adic solution for a proper $p$-normalised pair, it suffices, due to Lemma~\ref{Hensel}, to show that a non-singular solution exists. Using contractions as described in Section \ref{Coloured variables and Contractions}, this can be done by showing, that one can construct a primary variable at level $\tau+1$. 
        
        In the following there will be two different strategies to construct a primary variable at level at least $\tau+1$. For the first, one contracts the variables at level $0$ to primary variables at level at least $1$. Using contractions recursively, one can obtain  primary variables at higher levels, until one eventually reaches at least level $\tau+1$. 
	
        The second strategy will be used if $I_{0}^0\ge p^{\tau+1}+p^\tau-1$. By Lemma \ref{zeropresenting} with $\gamma=\tau+1$, it follows that the colour $0$ is zero-representing. In this case it suffices to have a contraction to a variable at level at least $\tau+1$, which can be traced back to at least one variable at level $0$ of a different colour than $0$. If such a variable can also be traced back to a variable at level $0$ of colour $0$, the variable is already primary. Else, there is a contraction to another variable at level at least $\tau+1$, using only the variables at level $0$ of colour $0$. Setting both of these variables $1$ and everything else zero proves, that there is a non-singular solution of $f= g = 0$.
	\begin{Definition}
        A variable which is either a variable at level $0$ of a different colour than $0$ or can be traced back to one is called \textit{colourful}.
	\end{Definition}
	\noindent Thus, if $I_0^0\ge p^{\tau+1}+p^\tau-1$, the goal is to create a colourful variable at level at least $\tau+1$.
	
	The gain of this second strategy are the variables at level $0$ of colour $0$. To contract primary variables at level at least $1$, one usually uses the variables at level $0$. If the goal is only to contract colourful variables at level at least $1$, it will suffice to use the $q_0$ variables at level $0$ which are colourful. Then, the variables at level $0$ of colour $0$ can be used to create variables at a higher level, to help contracting the colourful variables to colourful variables at an even higher level, until one eventually contracts them to a colourful variable at level at least $\tau+1$. This works, because then, one encounters one of the following two scenarios. Either the colourful variable at level at least $\tau+1$ can be traced back to a variable at level $0$ of colour $0$. Then one has used one of those variables, which were created using the variables at level $0$ of colour $0$, some way along the way, and the colourful variable at level at least $\tau+1$ is also primary. If on the other hand, the colourful variable at level at least $\tau+1$ cannot be traced back to a variable at level $0$ of colour $0$, those helpful variables were not needed, to create a colourful variable at level at least $\tau+1$. Hence, one can create a colourful variable at level at least $\tau+1$, without using any of the variables at level $0$ of colour $0$, which still enables one to create a variable at level at least $\tau+1$, using only those.
	
	The process of creating a colourful or primary variable at level at least $\tau+1$ will follow the same pattern. If one has a colourful or primary variable at level at least $l$, either this variable is already at level at least $l+1$, or one tries to find a contraction to a variable at level at least $l+1$, which contains the colourful or primary variable and thus ensures, that the resulting variable at level $l+1$ is colourful or primary, as well. To find such a contraction, one needs to guarantee, that there are other variables at the same level with certain properties. Thus, one differs between the colourful and primary variables, for which one only needs to know a lower bound of their level, and the remaining variables, which will be useful, to contract colourful or primary variables to colourful and primary variables at a higher level. For them it is important to know the precise level they are at. This will be considered by the following notation.
	
	A primary variable at level at least $l$ of colour nuance $\left(\nu, \mu\right)$ will be denoted by $P_{\nu\mu}^l$, whereas a colourful variable which otherwise has the same properties will be denoted by $C_{\nu\mu}^l$. The notation $E_{\nu \mu}^l$ will be used to describe a variable at the exact level $l$ of colour nuance $\left(\nu, \mu\right)$. Note that for $S \in \{C, P\}$ a variable of type $S_{\nu \mu}^l$ can either be of type $S_{\nu\mu}^{l+1}$ or of type $E_{\nu\mu}^l$, but not both. It will be said throughout the proof that a set of variables contracts to a variable with certain properties, if one the following cases occur. Either one of the variables in the set is already a variable with the desired properties, or the set of indices of these variables contains a contraction to a variable with these properties. This will help to minimize the amount of cases in which one has to distinguish between an $S_{\nu \mu}^l$ variables being of type $S_{\nu \mu}^{l+1}$ or $E_{\nu\mu}^l$ for $S \in \{C,P\}$. Sometimes one only wants to establish the level and the colour of one variable. Then, this is denoted by $P_\nu^l$, $C_\nu^l$ or $E_\nu^l$. If even the colour is of no importance, such a variable is said to be of type $P^l$, $C^l$ or $E^l$. In some cases, one has to denote, that a variable of type $E^l$ is not of colour $\nu$, or that a variable of type $E_{\nu}^l$ is not of colour nuance $\left(\nu, \mu\right)$. This is denoted by  $E_{\bar \nu}^l$ and $E_{\nu \bar{\mu}}^l$, respectively.
	
	It will turn out, that the number of $C^1$ and $P^1$ variables one can contract the $E^0$ variables to is at least partly dependent on the parameter $q_0$. Therefore, it will be useful to define a further parameter $r=r\left(f, g\right)$ for a pair $f, g$ which restricts the area for $q_0$ to
    \begin{align}
	\label{r}
	p^{\tau+1} +r p^\tau \le q_0 \le p^{\tau+1} +\left( r+1\right)  p^\tau -1.
	\end{align}
	For a proper $p$-normalised pair $f, g$ it follows that $r = r \left(f, g\right)\in \{-1, 0, 1, \dots, p-1\}$ due to $p^{\tau+1} -p^{\tau}+1\le q_0 \le 2p^{\tau+1} -2$.
        
    \section{Contraction Related Auxiliaries}
    \label{Contraction Related Auxiliaries}
    This section is a compilation of settings in which sets of variables contract to variables at a higher level.
    \subsection{Contracting One Specific Variable}
    The lemmata in this subsection describe situations in which one contracts sets of variables to one variable with specific properties.
	\begin{Lemma}
		\label{TAEL3}
		Let $\mathscr{K}$ be a set of indices of $E^l$ variables. If $|\mathscr{K}| \ge 2p-1$ and $q\left( \mathscr{K}\right) \ge p$, then $\mathscr{K}$ contains a contraction $J$ to a variable at level at least $l+1$, such that $J$ contains variables of at least two different colours.
	\end{Lemma}
    \begin{proof}
        This is a restatement of~\cite[Lemma 3]{twoaddeq}.
    \end{proof}
	\begin{Lemma}
    \label{p^n=2}
        Let $S \in \{C, P\}$. A set of $2p-1$ variables of type $S^l$ contracts to an $S^{l+1}$ variable.
    \end{Lemma}
    \begin{proof}
        Either one of the $S^l$ variables is already a variable of type $S^{l+1}$ or Lemma \ref{p^n} can be used with $n=2$ to show that the set of indices of the $2p-1$ variables of type $S^l$ contains a contraction to a variable at level at least $l+1$ which can be traced back to at least one of the $S^l$ variables. Therefore, it is an $S^{l+1}$ variable.
    \end{proof}
    \begin{Lemma}
        \label{FS1}
        Let $S \in \{C, P \}$ and let there be $3p-2$ variables of type $S^l$. Then one can contract them to a variable of type $S^{l+1}$, using at most $p$ of them.
	\end{Lemma}
    \begin{proof}
        Either one of the $S^l$ variables is already a variable of type $S^{l+1}$ or one can contract the $S^l$ variables to a variable at level at least $l+1$ using at most $p$ of them due to Lemma \ref{FS1C}. This variable can be traced back to at least one of the $S^l$ variables, thus it is an $S^{l+1}$ variable.
	\end{proof}
    \begin{Lemma}
		\label{contractiontolevell+1}
		Let there be $3p-2$ variables of type $E_\nu^l$ for $p\ge 5$ and $2p-1$ variables of type $E_\nu^l$ for $p=5$. Then one can contract at most $p$ of these variables to a variable of type $E^{l+1}$.
	\end{Lemma}
	\begin{proof}
        For $p\ge5$ see~\cite[Lemma 3.10]{GodSouNet} and for $p=5$ see~\cite[Lemma 3.8]{GundNet}.
	\end{proof}
    \begin{Lemma}
        \label{FS3}
        Let there be $3p-2$ variables of type $E_{\nu \mu}^l$ for $p \ge 5$ or $2p-1$ variables for $p=5$. Then one can contract at most $p$ variables to a variable of type $E_\nu^{l+1}$.
    \end{Lemma}
    \begin{proof}
        Let $\mathscr{K}$ be the set of indices of these variables. Let $c_i$ be the corresponding integer of the variable $x_i$. Due to Lemma \ref{FS4} for  $p\ge 5$ and Lemma \ref{2p-1fuerp=5} for $p=5$, there is a non-empty subset $J \subset \mathscr{K}$ with $|J|\le p$, such that $\sum_{j \in J} c_j \equiv 0 \bmod p$ while $\sum_{j \in J} c_j \nequiv 0 \bmod p^2$ and it follows that
        \begin{align*}
	          \sum_{j \in J}  \begin{pmatrix}\tilde a_j\\ \tilde b_j\end{pmatrix} \equiv \sum_{j \in J} c_j  \left(\boldsymbol e_\nu + \mu \boldsymbol e^\nu \right) \equiv \left(\boldsymbol e_\nu + \mu \boldsymbol e^\nu\right) \sum_{j \in J} c_j\nequiv 0  \mod p^2,
        \end{align*}
        while $\sum_{j \in J}  c_j \equiv 0 \bmod p$. As $p\mid \boldsymbol e^\nu$, this leaves
        \begin{align*}
            \sum_{j \in J}  \begin{pmatrix}\tilde a_j\\  \tilde b_j\end{pmatrix}  \equiv \boldsymbol e_\nu \sum_{j \in J} c_j \equiv p c \boldsymbol e_\nu \mod p^2
        \end{align*}
        for some $c$ not congruent to $0$ modulo $p$. Hence, by setting $x_i=1$ for all $i \in J$, one can see that $J$ is a contraction of at most $p$ variables to a variable of type $E_\nu^{l+1}$.
    \end{proof}
    \begin{Lemma}
		\label{createavariableinm_1-I_0^1}
		Let there be $p-1$ variables of type $E_{\nu \mu_1}^l$ and one of type $E_{\nu \mu_2}^l$ with $\mu_1 \neq \mu_2$. Then one can contract them to an $E_{\bar \nu}^{l+1}$ variable.
	\end{Lemma}
	\begin{proof}
        Define $x^{-1}$ for an integer $x \in \Z\backslash p\Z$ as the element in $\{1, \dots, p-1\}$ which solves $x \cdot x^{-1} \equiv 1 \bmod p$.
	
        Let $\mathscr{K}$ be the set of indices of those $p$ variables and $c_i$ be the corresponding integer for $i \in \mathscr{K}$.
		Let $x_{i_0}$ be the $E_{\nu\mu_2}^l$ variable. Due to Lemma \ref{nicht0} there is a solution of 
		\begin{align*}
			\sum_{i \in \mathscr{K}} c_i y_i^k \equiv t p \mod p^2
		\end{align*}
		for some $t \in \{ 1, \dots, p\}$ with $y_{i_0} \nequiv 0 \bmod p$. Consequently, one has $y_{i_0}^k \equiv 1 \bmod p$ because $p-1 \mid k$ and it follows that 
		\begin{align*}
			\sum_{i \in \mathscr{K}}  \begin{pmatrix}\tilde a_i\\ \tilde b_i\end{pmatrix} y_i^k \equiv \sum_{i \in \mathscr{K}\backslash \{i_0\}}  c_i \left(\boldsymbol{e}_{\nu} + \mu_1 \boldsymbol{e}^{\nu} \right) y_i^k + c_{i_0}  \left( \boldsymbol{e}_{\nu} + \mu_2 \boldsymbol{e}^{\nu} \right) y_{i_0}^k \equiv t p \boldsymbol{e}_\nu + c_{i_0} \boldsymbol{e}^\nu \left(\mu_2 - \mu_1\right)\mod p^2,
		\end{align*} 
		which is divisible by $p$ because $\boldsymbol{e}^\nu$ is. 
		For $\nu=0$ one has
		\begin{align*}
			tp \boldsymbol{e}_\nu + c_{i_0} \boldsymbol{e}^\nu \left(\mu_2 - \mu_1\right) &\equiv	p \left( t \begin{pmatrix}1\\0\end{pmatrix} + c_{i_0} \begin{pmatrix}0\\1\end{pmatrix} \left( \mu_2 - \mu_1 \right)\right)\\
			&\equiv p \left( c_{i_0} \left( \mu_2-\mu_1\right) \begin{pmatrix}t c_{i_0}^{-1} \left( \mu_2-\mu_1\right)^{-1}\\1\end{pmatrix}\right) \mod p^2
		\end{align*}
		because $p$ divides neither $c_{i_0}$ nor $\mu_2-\mu_1$. It follows that the resulting variable lies at level $l+1$ and is of colour $\nu' \neq 0$ with $\nu' \equiv t c_{i_0}^{-1} \left( \mu_2-\mu_1\right)^{-1} \bmod p$.
		For $\nu\neq 0$ one gets
		\begin{align*}
			tp \boldsymbol{e}_\nu + c_{i_0} \boldsymbol{e}^\nu \left(\mu_2 - \mu_1\right) \equiv	p \left( t \begin{pmatrix}\nu\\1\end{pmatrix} + c_{i_0} \begin{pmatrix}1\\0\end{pmatrix} \left( \mu_2 - \mu_1 \right)\right)\mod p^2
		\end{align*}
		which is for $t\equiv 0 \bmod p$  congruent to
		\begin{align*}
			p \left( c_{i_0} \left( \mu_2 - \mu_1 \right)\begin{pmatrix}1\\0\end{pmatrix} \right)
		\end{align*}
		and else congruent to
		\begin{align*}
			p \left( t\begin{pmatrix}\nu +t^{-1} c_{i_0} \left( \mu_2-\mu_1\right)\\1\end{pmatrix}\right).
		\end{align*}
		Hence, again because $p$ divides neither $c_{i_0}$ nor $\mu_2-\mu_1$ one obtains a variable at level $l+1$, which is for $t\equiv 0 \bmod p $ of colour $0$ and for $t \nequiv 0 \bmod p$ of colour $\nu'$ for $\nu' \equiv \nu + t^{-1} c_{i_0} \left( \mu_2-\mu_1\right) \bmod p$ with $\nu' \neq \nu$.
	\end{proof}
	\begin{Lemma}
		\label{onecolourfulandp-1differentvariablesofthesamecolour}
		Let $S \in \{C, P\}$ and $0 \le m \le p-1$. Let there be $p-m-1$ variables of type $E_\nu^l$ and $m+1$ of type $S_\nu^l$. Then they contract to a variable of type $S^{l+1}$. 
	\end{Lemma}
	\begin{proof}
        Either one of the $S_\nu^l$ variables is already a $S_\nu^{l+1}$ variable, or one can assume, that they are all of type $E_\nu^l$ as well. The cases $l>0$ can be reduced to the case $l=0$ by working with the level coefficient vector $\binom{\tilde a_i}{\tilde b_i}$ instead of the coefficient vector $\binom{a_i}{b_i}$.
        See~\cite[Lemma 3.7]{GundNet} for the case $l=0$.
	\end{proof}
	\begin{Lemma}
        \label{FS5}
        Let $\mathscr H$ be a set of indices of variables of type $E_\nu^l$ with $|\mathscr H | \ge 4p-3$ and either for all $i \in \mathscr H$ the corresponding integer $c_i$ is congruent to an element in the set $\left\{1, 2, \dots, \frac{p-1}{2} \right \}$ modulo $p$ or all $c_i$ are congruent to elements in the set $\left \{ \frac{p+1}{2}, \dots, p-1\right \}$. Then $\mathscr{H}$ contains a contraction $\mathscr{K}$ to a variable of type $E_\nu^{l+1}$, with $|\mathscr{K}| \le 2p-2$.
    \end{Lemma}
    \begin{proof}
        For all $i \in \mathscr{H}$, let $\left( \nu, \mu_i\right)$ be the colour nuance of the variable $x_i$ and let $d_i \in \left \{1, 2, \dots, p-1\right \}$ and $f_i \in \left \{ 0, 1, \dots, p-1 \right \}$ be such that as $c_i=d_i + p f_i$.
    	
        For the proof one can assume that $|\mathscr H | = 4p-3$. If this is not the case, one can take a subset of $\mathscr H$ to obtain the desired result. The first part proves the weaker claim that $\mathscr{H}$ contains a subset $\mathscr{K}$ containing at most $2p$ variables such that
        \begin{align*}\sum_{i \in J} \begin{pmatrix} d_i \\ d_i \mu_i \end{pmatrix}\equiv 0 \mod p \quad \textup{and} \quad \sum_{i \in J} \begin{pmatrix} \tilde{a}_i\\ \tilde{b}_i \end{pmatrix} \equiv dp \boldsymbol e_{\nu} \mod p^2,
        \end{align*}
        for some $d \nequiv 0 \bmod p$.
        By Lemma \ref{p^n}, the set $\mathscr H$ contains a non-empty subset $J$ such that
        \begin{align}
        \label{1trick}
            \sum_{i \in J} \begin{pmatrix}
                                            d_i\\ d_i \mu_i \\f_i  \\ 1 
                                         \end{pmatrix}\equiv 0 \mod p.
        \end{align}
        This leads to
        \begin{align*}
            \sum_{i \in J} \begin{pmatrix}
                            \tilde a_i \\ \tilde b_i 
                           \end{pmatrix}
                           &\equiv \sum_{i \in J} c_i\left(\boldsymbol e_\nu + \mu_i \boldsymbol e^{\nu} \right) \equiv \sum_{i \in J} \left(d_i + f_i p\right) \left(\boldsymbol e_\nu + \mu_i \boldsymbol e^{\nu}\right) \\ &\equiv  \sum_{i \in J} d_i \boldsymbol e_\nu +  \sum_{i \in J} d_i \mu_i \boldsymbol e^{\nu} + \sum_{i \in J} f_i p \boldsymbol e_\nu + \sum_{i \in J} f_i p \mu_i \boldsymbol e^{\nu}\\
                           & \equiv   \boldsymbol e_\nu \sum_{i \in J} d_i  \mod p^2,
        \end{align*}
        where the last equivalence holds due to $p \mid \boldsymbol{e}^\nu$ and the second and third entry in $\eqref{1trick}$. The first entry shows that this is congruent to $0$ modulo $p$. As $J$ is a non-empty subset of $\mathscr H$, it follows from the fourth entry, that $|J|\in \left \{ p, 2p, 3p \right \}$.
        If $|J| = 3p$, take a subset $\tilde J \subset J$ containing $3p-2$ elements. By Lemma \ref{p^n} with $n=3$, there is a subset $\hat J \subseteq \tilde J$ with 
        \begin{align*}
            \sum_{i \in \hat J} \begin{pmatrix}
                                            d_i\\  d_i\mu_i\\ f_i 
                                         \end{pmatrix}\equiv 0 \mod p,
        \end{align*}
        and hence, 
        \begin{align*}
            \sum_{i \in \hat J} \begin{pmatrix}
                            \tilde a_i \\ \tilde b_i 
                           \end{pmatrix} \equiv  \boldsymbol e_\nu  \sum_{i \in \hat J} d_i \mod p^2,
        \end{align*}
        as before, which again is congruent to $0$ modulo $p$. As $J= \hat{J} \cup \left( J\backslash \hat J\right)$, it follows that
        \begin{align*}
         \sum_{i \in J \backslash \hat J} \begin{pmatrix} d_i \\d_i \mu_i \\ f_i \end{pmatrix} \equiv 0 \mod p,                 
        \end{align*}                  
        and therefore,
        \begin{align*}
             \sum_{i \in J \backslash \hat J} 
           \begin{pmatrix}
                            \tilde a_i \\ \tilde b_i 
                           \end{pmatrix} \equiv \boldsymbol e_\nu \sum_{i \in J\backslash \hat J} d_i \mod p^2,
        \end{align*}
        which is congruent to $0$ modulo $p$ as well.
        Furthermore, both sets $\hat J$ and $J \backslash \hat J$ are non-empty, and the smallest of them has at most $\frac{3p}{2}\le 2p$ elements. It follows, that in every case there is a non-empty set $\mathscr{K}\subset \mathscr H$ containing at most $2p$ elements, such that 
        \begin{align*}           
            \sum_{i \in  J} \begin{pmatrix}
                            \tilde a_i \\ \tilde b_i 
                           \end{pmatrix} \equiv \boldsymbol e_\nu \sum_{i \in J} d_i  \mod p^2,  \qquad \textup{ and }\qquad
            \sum_{i \in J } \begin{pmatrix} d_i \\ d_i \mu_i \end{pmatrix}  \equiv 0 \mod p.
        \end{align*}
        Assume now for such a set $\mathscr{K}$ that all corresponding integers $c_i$ are congruent to elements in the set $\left \{ 1, 2, \dots, \frac{p-1}{2}\right \}$ modulo $p$. It follows, that $d_i$ lies in the same set for all $i \in \mathscr{K}$. Hence, it can be deduced from
        \begin{align*}
            1 \le \sum_{i \in \mathscr{K}} d_i \le \sum_{i \in \mathscr{K}} \frac{p-1}{2} \le p\left(p-1\right),
        \end{align*}
        that $\sum_{i \in \mathscr{K}} d_i \nequiv 0 \bmod p^2$ and therefore,
        \begin{align*}
            \sum_{i \in \mathscr{K}} \begin{pmatrix}\tilde a_i\\ \tilde b_i\end{pmatrix} \equiv d p \boldsymbol e_\nu \mod p^2 
        \end{align*}
        for some $d \nequiv 0 \bmod p$. This proves the weaker claim if all $c_i$ are modulo $p$ congruent to an element in the set $\{1, \dots, \frac{p-1}{2}\}$. 
        Now let all $c_i$ be congruent to elements in the set $\left \{ \frac{p+1}{2}, \dots, p-1 \right \}$. It follows that
        \begin{align*}
            \begin{pmatrix}- \tilde a_i\\- \tilde b_i\end{pmatrix} \equiv \left(p^2-c_i\right)\left(\boldsymbol e_\nu + \mu_i     \boldsymbol e^{\nu}\right) \equiv \left(p -d_i + p \left(p-f_i-1\right)\right)\left(\boldsymbol e_\nu + \mu_i     \boldsymbol e^{\nu}\right)\mod p^2,
        \end{align*}
        and, the corresponding integers $p -d_i + p \left(p-f_i-1\right)$ lie modulo $p$  in $\left \{1, 2, \dots, \frac{p-1}{2}\right \}$, again. Using the obtained results, there is a subset $\mathscr{K}\subset \mathscr H$ with $|\mathscr{K}| \le 2p$ and
        \begin{align*}
			\sum_{j \in \mathscr{K}} \begin{pmatrix}-\tilde a_j \\-\tilde b_j\end{pmatrix} \equiv d p\boldsymbol e_\nu \mod p^2
        \end{align*}
        for some $d \nequiv 0 \bmod p$ and, as $\binom{\tilde a_i}{\tilde b_i}$ lies in the same set $\mathscr{L}_{\nu \mu}$ as $\binom{-\tilde a_i}{-\tilde b_i}$, one further has
        \begin{align*}
            \sum_{j\in \mathscr{K}} \begin{pmatrix}p-d_j\\\left(p-d_j\right)\mu_j\end{pmatrix} \equiv 0 \mod p.
        \end{align*}
        It follows that
        \begin{align*}
        \sum_{j \in \mathscr{K}} \begin{pmatrix}\tilde a_j\\\tilde b_j\end{pmatrix} = -\sum_{j \in \mathscr{K}} \begin{pmatrix}-\tilde a_j\\-\tilde b_j\end{pmatrix} \equiv - d p\boldsymbol e_\nu \mod p^2
        \end{align*}
        for some $d \nequiv 0 \bmod p$ and it further holds that
        \begin{align*}
            \sum_{j\in \mathscr{K}} \begin{pmatrix}d_j\\d_j\mu_j\end{pmatrix} \equiv 0 \mod p.
        \end{align*}
        This completes the proof for the weaker claim.
        Now let $\mathscr{K}\subset \mathscr H$ be a subset with $|\mathscr{K}| \le 2p$,
        \begin{align*}
            \sum_{i \in \mathscr{K}}
        \begin{pmatrix} d_i \\ d_i \mu_i
        \end{pmatrix}\equiv 0 \mod p \quad \textup{and} \quad \sum_{i \in \mathscr{K}} \begin{pmatrix} \tilde a_i\\ \tilde b_i \end{pmatrix} \equiv pd\boldsymbol e_\nu \mod p^2
        \end{align*}
        for some $d \nequiv 0 \bmod p$.
        Assuming that $|\mathscr{K}| \ge 2p-1$, there is, according to Lemma \ref{p^n} with $n=2$, a subset $\tilde{\mathscr{K}} \subset \mathscr{K}$ with $|\tilde{\mathscr{K}}| \le 2p-1$ and
        \begin{align*}
            \sum_{i \in \tilde{\mathscr{K}}} \begin{pmatrix}d_i\\d_i  \mu_i\end{pmatrix} \equiv 0 \mod p.
        \end{align*}
        It follows, that
        \begin{align*}
            \sum_{i \in \tilde{\mathscr{K}}} \begin{pmatrix}\tilde a_i\\ \tilde b_i\end{pmatrix} \equiv \boldsymbol e_\nu \sum_{i \in \tilde{\mathscr{K}}} d_i+ p \boldsymbol{e}_\nu \sum_{i \in \tilde{\mathscr{K}}} f_i \mod p^2,
        \end{align*}
        which is congruent to $0$ modulo $p$, but not necessarily incongruent to $0$ modulo $p^2$. 
        As 
        \begin{align*}
            \sum_{i \in \mathscr{K}\backslash \tilde{\mathscr{K}}} \begin{pmatrix}d_i\\d_i \mu_i\end{pmatrix} \equiv 0 \mod p 
        \end{align*}
        holds as well, one can deduce, that
        \begin{align*}
            \sum_{i \in \mathscr{K} \backslash \tilde{\mathscr{K}}} \begin{pmatrix}\tilde a_i\\\tilde b_i\end{pmatrix} \equiv \boldsymbol e_\nu \sum_{i \in \mathscr{K} \backslash \tilde{\mathscr{K}}} d_i +p \boldsymbol e_\nu \sum_{i \in \mathscr{K} \backslash \tilde{\mathscr{K}}} f_i \mod p^2,
        \end{align*}
        which is again congruent to $0$ modulo $p$. 
        For at least one of those sets, either $\tilde{\mathscr{K}}$ or $ \mathscr{K} \backslash \tilde{\mathscr{K}}$, the sum is not congruent to $0$ modulo $p^2$ as the sum over all $i \in \mathscr{K}$ is not, and therefore, it is impossible for both subsums to be congruent to $0$ modulo $p^2$. The set for which this sum is incongruent to $0$ modulo $p^2$ is a contraction to a variable of type $E_\nu^{l+1}$.
        
        Both subsets are non-empty and hence, as all $d_i$ are incongruent to $0$ modulo $p$, they contain at least $2$ elements. Thus, each one has a most $2p-2$ elements, which proves the claim.
     \end{proof}
     \begin{Lemma}
		\label{2p-1}
		Let $S \in \{C, P\}$ and $0\le m \le p-1$. Let there be $p+m$ variables of type $S^l$ and further $p-m-1$ variables of type $E_\nu^l$. Then one can contract them to an $S^{l+1}$ variable.
	\end{Lemma}
	\begin{proof}
        If one of the $S^l$ variables is already an $S^{l+1}$ variable, the claim is fulfilled. Thus, one can assume, that the $S^l$ variables are $E^l$ variables as well. If there are $p$ variables of the same colour $\mu$, then at least one of them is an $S^l$ variables, because there are at most $p-1$ variables which are not. Hence, Lemma \ref{onecolourfulandp-1differentvariablesofthesamecolour} shows that one can contract them to an $S^{l+1} $ variable.
        
        Else, there are at most $p-1$ variables of the same colour. Let $\mathscr{K}$ be the set of indices of all $2p-1$ variables. Then, one has $I_\mx\left( \mathscr{K}\right) \le p-1$, and thus, $q\left( \mathscr{K} \right) \ge p$. By Lemma \ref{TAEL3}, the set $\mathscr{K}$ contains a contraction to  to a variable at level at least $l+1$, using at least two different colours. One can trace that variable back to at least one of the $S^l$ variables, because the variables which are not of type $S^l$ are all of the same colour, which proves the claim.
	\end{proof}
	\begin{Lemma}
		\label{liftingonespecialvariablehigher}
		Let $S \in \{C, P\}$ and $ 0\le m \le p-1$. Let there be $p-1$ variables of type $E_\nu^l$, $p-m-1$ variables of type $E_{\bar \nu}^l$ and $m+1$ variables of type $S^l$. Then one can contract them to an $S^{l+1}$ variable.
	\end{Lemma}
	\begin{proof}
        If one of the variable of type $S^l$ is already an $S^{l+1}$ variable, the claim is fulfilled, thus one can assume that these variables are of type $E^l$ as well. Furthermore, one can assume, that none of the $S^l$ variables is of type $S_\nu^l$, because else, Lemma \ref{onecolourfulandp-1differentvariablesofthesamecolour} can be use to contract the $p-1$ variables of type $E_\nu^l$ together with the $S_\nu^l$ variable to an $S^{l+1}$ variable.
        
        Therefore, one can assume that one has $p-1$ variables of type $E_\nu^l$ and $p$ variables of type $E_{\bar \nu}^l$ from which at least one is an $S^l$ variable. For convenience name the $E_\nu^l$ variables $x_1, \dots, x_{p-1}$ and the $E_{\bar\nu}^l$ variables $x_p, \dots, x_{2p-1}$, where $x_{2p-1}$ is an $S^l$ variable. Furthermore, let $c_i$ be the corresponding integer of $x_i$ for $1 \le i \le 2p-1$ and $\nu_i\neq \nu$ the colour of the variables $x_i$ for $p \le i \le 2p-1$. These $2p-1$ variables contract to an $S^{l+1}$ variable if there is a solution of
        \begin{align*}
            \sum_{i=1}^{p-1} c_i \boldsymbol{e}_{\nu} x_i^k + \sum_{i=p}^{2p-1} c_i \boldsymbol{e}_{\nu_i}x_i^k \equiv 0 \mod p,
        \end{align*}
        with $x_{2p-1} \nequiv 0 \bmod p$. The existence of such a solution follows from the proof of Theorem $2$ by Olson and Mann~\cite{ppolsonmann}, but not from the statement of the theorem, from which one can only conclude the existence of a solution, but not that one has one with $x_{2p-1}\nequiv 0 \bmod p$. Thus, for the convenience of the reader, the following contains a proof that such a solution exists. In essence the proof uses the same methods as the proof by Olson and Mann, but is tailored for this exact case.
        
        By applying the linear transformation induced by
        \begin{align*}
            \begin{pmatrix}
             1&0\\
             1&-\nu
            \end{pmatrix}
        \end{align*}
        if $\nu \neq 0$, one can transform the case $\nu\neq 0$ to the case $\nu=0$, because
        \begin{align*}
            \begin{pmatrix}
             1&0\\
             1&-\nu
            \end{pmatrix} \boldsymbol{e}_\nu= \nu \boldsymbol{e}_0  \qquad \textup{and} \qquad \begin{pmatrix}
             1&0\\
             1&-\nu
            \end{pmatrix} \boldsymbol{e}_{\nu_i}\in \mathscr{L}_{\tilde \nu}
        \end{align*}
        for some $\tilde \nu \neq \nu$.
        All that remains is to solve a system of the kind
        \begin{align}
        \label{nu=0}
            \sum_{i=1}^{p-1} \begin{pmatrix}\alpha_i\\0\end{pmatrix} x_i^k + \sum_{i=p}^{2p-1} \begin{pmatrix}\beta_i\\ \gamma_i\end{pmatrix}x_i^k \equiv 0 \mod p
        \end{align}
        where $p \nmid \alpha_i$ for $1\le i \le p-1$ and $p \nmid \gamma_i$ for $p \le i \le 2p-1$ such that $p \nmid x_{2p-1}$. By Lemma \ref{nicht0}, there is a solution $y_i$ with $p\le i \le 2p-1$ of the equation
        \begin{align*}
            \sum_{i=p}^{2p-1} \gamma_i y_i^k \equiv 0 \mod p
        \end{align*}
        with $y_{2p-1} \nequiv 0 \bmod p$. This reduces the system \eqref{nu=0} by setting $x_i=y_i$ for $p \le i \le 2p-1$ to
        \begin{align}
        \label{C}
            \sum_{i=1}^{p-1} \alpha_i x_i^k+ C\equiv 0 \mod p
        \end{align}
        for $C= \sum_{i=p}^{2p-1} \beta_i y_i^k$. Now consider an additional variable $y_0$. If $p \nmid y_0$ then $y_0^k \equiv 1 \bmod p$, hence, applying Lemma \ref{nicht0} again, this time to the system
        \begin{align*}
            \sum_{i=1}^{p-1} \alpha_i x_i^k+ Cy_0^k\equiv 0 \mod p
        \end{align*}
        provides a solution $y_i$ with $p \nmid y_0$. It follows that $x_i=y_i$ for $1\le i \le p-1$ is also a solution for \eqref{C}, and therefore, one has a solution of \eqref{nu=0} given by $x_i=y_i$ with $1 \le i \le 2p-1$ with $p \nmid x_{2p-1}$. This completes the proof.
	\end{proof}
	\subsection{Contracting Several Variables}
	The lemmata in this section show how to contract a set of variables at level at least $l$ to another set of variables at level at least $l+1$.
	\begin{Lemma}
		\label{5.1}
		Let $\mathscr{H} \subset \{ 1, \dots, m_0\}$ be a subset of indices of variables at level $0$. Then $\mathscr{H}$ contains at least
		\begin{align*}
			\textup{min} \left( \left \lfloor  \frac{|\mathscr{H}|}{2 p-1 }\right \rfloor , \left \lfloor \frac{q \left( \mathscr{H}\right)}{p} \right \rfloor \right)
		\end{align*}
		pairwise disjoint contractions to variables of type $P^1$.
	\end{Lemma}
	\begin{proof}
		This is the special case $\delta =\textup{gcd}\left(k, p-1\right)=p-1$ of a result from Lemmata 1 and 3 of~\cite{twoaddeq} which is proved in the second paragraph of Section 6 of that paper.
	\end{proof}
	\begin{Lemma}
        \label{FS2}
        Let $S \in \{C, P\}$ and let there be $x$ variables of type $S^l$. They contract to $\left \lceil \frac{x+3}{p} \right \rceil -3$ variables of type $S^{l+1}$, where each contraction contains at most $p$ variables, leaving at least $\min\{2p-2,x\}$ variables of type $S^{l}$ unused.
        \end{Lemma}
    \begin{proof}
        For $x\le 3p-3$ the statement is trivial. Therefore, let $x \ge 3p-2$. Assume first, that all $x$ variables are also of type $E^l$. Then there is a contraction of at most $p$ variables to an $S^{l+1}$ variable due to Lemma \ref{FS1}. Hence, after doing this $\left \lceil \frac{x+3}{p} \right \rceil -4$ times, there are still at least
        \begin{align*}
         x-\left(\left \lceil \frac{x+3}{p} \right \rceil -4\right)p\ge x-\left(x+3+p-1-4p\right)=3p-2
        \end{align*}
        unused $S^l$ variables. Hence, one can apply Lemma \ref{FS1} once more, to obtain $\left \lceil \frac{x+3}{p}\right \rceil -3$ contractions, leaving at least $2p-2$ variables unused. Thus, in this case, the claim holds.
        
        Now assume that of the $x$ variables of type $S^{l}$ there are $y$ variables already of type $S^{l+1}$ while the remaining $x-y$ variables are of type $E^l$.
        One has
        \begin{align*}
            y \ge \left \lceil \frac{x+3}{p} \right \rceil -3+2p-2-\left(x-y\right)
        \end{align*}
        because of $x \ge 3p-2$. If $x-y \le 2p-2$, one can divide the $y$ variables of type $S^{l+1}$ in one set containing $\left \lceil \frac{x+3}{p} \right \rceil -3 $ and one set containing $2p-2-\left(x-y\right)$ of them. The variables in the second set together with the remaining $x-y$ variables of type $S^l$ are at least $2p-2$ variables of type $S^l$, while the first set contains the $\left \lceil \frac{x+3}{p} \right \rceil -3 $ variables of type $S^{l+1}$.
        Thus one can assume, that $x-y \ge 2p-1$ and use the first part of this proof. The set of the $x-y$ variables of type $E^l$ contains at least
        \begin{align*}
            \left \lceil \frac{x-y+3}{p} \right \rceil -3
        \end{align*}
        contractions to variables of type $S^{l+1}$, leaving at least $2p-2$ variables of type $S^l$ unused. Together with the $y$ variables of type $S^{l+1}$ this gives at least
        \begin{align*}
            \left \lceil \frac{x-y+3}{p} \right \rceil -3 +y= \left \lceil \frac{x-y+3}{p} + y \right \rceil -3 = \left \lceil \frac{x+y\left(p-1\right)+3}{p} \right \rceil -3 \ge \left \lceil \frac{x+3}{p} \right \rceil -3
        \end{align*}
        to variable of type $S^{l+1}$.
    \end{proof}
	 \begin{Lemma}
		\label{FS6klein}
		Let there be $x$ variables of type $E_\nu^l$. They contract to $\left \lceil \frac{x}{2p-2} \right \rceil -4$ variables of type $E_\nu^{l+1}$, leaving at least $\min \{6p-9,x\}$ variables of type $E_\nu^l$ unused.
	\end{Lemma}
	\begin{proof}
        For $x < 8p-7$ the statement is trivial. If $x\ge 8p-7$, one can divide the $x$ variables in two sets. Those for which the corresponding integer $c_i$ is congruent to one element in $\{ 1, \dots, \frac{p-1}{2}\}$ modulo $p$, and the remaining variables. As long as there are at least $8p-7$ variables left, at least one of these sets contains at least $4p-3$ variables, which indicates that one can contract at most $2p-2$ of them to a variable of type $E_\nu^{l+1}$ due to Lemma \ref{FS5}. Doing this $\left \lceil \frac{x}{2p-2} \right \rceil -5$ times leaves at least
		\begin{align*}
		x- \left( 2p-2\right) \left( \left \lceil \frac{x}{2p-2} \right \rceil -5\right) \ge x-x-2p+3+ 10p-10= 8p-7
		\end{align*}
		unused variables, hence, there is another contraction, leaving at least $6p-9$ variables unused.
	\end{proof}
    \begin{Lemma}
    	\label{FS6}
        A set of $x\ge 3p^2-3p+1$ variables of type $E_\nu^l$ contracts to $\left \lceil \frac x p \right \rceil - 2p+ \frac{p-3}{2}$ variables of type $E_\nu^{l+1}$ for $p \ge 5$. A set of $x \ge 2p^2-2p+1$ variables of type $E_\nu^l$ contracts to $\left \lceil \frac x p \right \rceil - 2p+ 3$ variables of type $E_\nu^{l+1}$ for $p=5$. In both cases, this leaves at least $6p-9$ of the $E_\nu^l$ variables unused.
    \end{Lemma}
    \begin{proof}
        A set of at least $\left(3p-3\right)p +1$ variables of type $E_\nu^l$ contains at least $3p-2$ variables which are of the same colour nuance. By Lemma \ref{FS3}, one can contract at most $p$ variables of them to a variable of type $E_\nu^{l+1}$. Repeating this as often as possible provides $\left \lceil \frac  x p \right \rceil - 3p +3$ variables of type $E_\nu^{l+1}$ and leaves at least 
        \begin{align*}
            x- p \left(\left \lceil \frac  x p \right \rceil - 3p +3\right)\ge x-\left(x+ p -1 -3 p^2+3p\right)= 3p^2-4p+1
        \end{align*}
        unused $E_\nu^l$ variables. 
        For $p=5$ this can be done as long as there are at least $\left( 2p-2\right)p +1$ variables left. Therefore, one can do it $\left \lceil \frac  x p \right \rceil - 2p +2$ times, leaving at least
        \begin{align*}
        	x- p \left(\left \lceil \frac  x p \right \rceil - 2p +2\right)\ge x-\left(x+ p -1 -2 p^2+2p\right)= 2p^2-3p+1
        \end{align*}
        unused variables.
        Using Lemma \ref{FS6klein} provides another $p+\frac{p-1}{2}-4$ variables of type $E_\nu^{l+1}$ for $p\ge5$ and one for $p=5$, while leaving at least $6p-9$ unused variables. All in all, one obtains
        \begin{align*}
           \left \lceil \frac x p \right \rceil -3p +3 + p + \frac{p-1} 2 -4= \left \lceil \frac x p \right \rceil -2p + \frac{p-3} 2 
        \end{align*}
        variables of type $E_\nu^{l+1}$ for $p\ge 5$ and
        \begin{align*}
            \left \lceil \frac x p \right \rceil -2p +2 +1= \left \lceil \frac x p \right \rceil -2p +3
        \end{align*}
        for $p=5$.
    \end{proof}
    \begin{Lemma}
		\label{x-y-Lemma}
		Let $S \in \{C, P\}$ and $x$, $y$ and $z$ be non-negative integers with $y+z \ge \left(2-m\right) p-2$ for some $m \in \{0,1,2\}$ and $x-m \ge 0$. Let there be $\left(p-1\right)y$ variables of type $E_\nu^l$, $\left(p-1\right)y$ variables of type $E_{\bar \nu}^l$ and $px+y+z$ variables of type $S^l$. Then one can contract them to $x+y-m$ variables of type $S^{l+1}$ without using $z+mp$ of the variables of type $S^l$.
	\end{Lemma}
	\begin{proof}
		Using Lemma \ref{FS1} to contract $p$ of the variables of type $S^l$ to an $S^{l+1}$ variable can be done $x-m$ times. This leaves $y+z+mp\ge 2p-2$ variables of type $S^{l}$. Then, one can construct $y$ sets, each consisting of one $S^l$ variable, $p-1$ variables of type $E_\nu^l$ and $p-1$ variables of type $E_{\bar \nu}^l$. By Lemma \ref{liftingonespecialvariablehigher}, each of this sets contains a contraction to an $S^{l+1}$ variable, giving a total of $x+y-m$ variables of type $S^{l+1}$ as claimed, without using $z+mp$ variables of type $S^l$.
	\end{proof}
    \begin{Lemma}
        \label{neu}
        Let $S \in \{C, P\}$ and $x$ be a non-negative integer. Let $\mathscr{K}$ be a set of $E^l$ variables with $| \mathscr{K}| \ge \left( 2p-2\right)x + p^2-3p+1$ and $q \left( \mathscr{K} \right) \ge \left(p-1\right) x$ and let there be further $x$ variables of type $S^l$. Then one can contract them to $x$ variables of type $S^{l+1}$.
	\end{Lemma}
	\begin{proof}
        The first part of the proof will show via induction on $x$ that the set $\mathscr{K}$ contains $x$ distinct sets $S_i$ with $|S_i|=2p-2$ and $q\left(S_i\right)=p-1$ for all $1 \le i \le x$.
        
        For $x=0$ the statement is true. 
        It suffices to show for $x \ge 1$ that $\mathscr{K}$ contains a set $\mathscr{H}$ with $|\mathscr{H}| =2p-2$ and $q\left( \mathscr{H} \right)=p-1$ such that $|\mathscr{K}\backslash \mathscr{H}|\ge \left(x-1\right) \left(2p-2\right) +p^2-3p+1$ and $q \left( \mathscr{K}\backslash \mathscr{H} \right)\ge \left(x-1\right) \left(p-1\right)$. If such a set $\mathscr{H}$ exists, the induction hypothesis ensures that one can find further $x-1$ distinct sets in $\mathscr{K}\backslash \mathscr{H}$.
        
        Let $|\mathscr{K}| = x \left(2p-2\right) + p^2-3p+1 +\alpha$ and $q\left(\mathscr{K}\right) = x \left(p-1\right) + \beta$ with $\alpha, \beta \in \N_0$. As $x \ge 1$ it follows that $q\left(\mathscr{K}\right) \ge p-1$ and $|\mathscr{K}| \ge p^2-p-1 = \left(p+1\right) \left( p-2\right) +1$, hence, $I_\mx\left(\mathscr{K} \right)=I_\nu\left(\mathscr{K}\right) \ge p-1$ for some $0 \le \nu \le p$. Thus, one can take $\mathscr{H}$ as a set containing $p-1$ variables of type $E_\nu^l$ and $p-1$ variables of type $E_{\bar \nu}^l$ from which it follows that $|\mathscr{H}|=2p-2$, $q \left(\mathscr{H}\right)=p-1$ and
        \begin{align*}
            |\mathscr{K}\backslash \mathscr{H}| = |\mathscr{K}|- 2p +2 \ge \left(x -1\right) \left(2p-2\right) + p^2-3p+1.
        \end{align*}
        For $\beta \ge p-1$ one has the trivial bound
        \begin{align*}
            q\left(\mathscr{K}\backslash \mathscr{H}\right) \ge q \left(\mathscr{K}\right) -2 \left(p-1\right)=\left(x-1\right) \left(p-1\right)+ \beta - \left(p-1\right) \ge \left(x-1\right) \left(p-1\right),
        \end{align*}
        whereas for $\beta \le p-2$ it follows that
        \begin{align*}
            I_\mx\left(\mathscr{K} \right)&= |\mathscr{K}|-q\left(\mathscr{K}\right)= x \left(p-1\right) + \beta + \alpha + p^2-3p+1- 2 \beta \\
            &\ge q\left( \mathscr{K}\right) + p^2-5p+5  \ge q\left(\mathscr{K}\right)
        \end{align*}
        and thus
        \begin{align*}
            q\left(\mathscr{K}\backslash \mathscr{H}\right) = q\left(\mathscr{K}\right)-\left(p-1\right) \ge \left(x-1\right) \left(p-1\right).
        \end{align*}
        It follows, that the set $\mathscr{K}$ contains $x$ distinct sets $S_i$ with $|S_i|=2p-2$ and $q\left(S_i\right)=p-1$.
        
        For each set $S_i$ there is a $\nu_i$ such that $I_\mx\left(S_i\right) = I_{\nu_i}\left(S_i\right)=p-1$. For $i \in \{1, \dots, x\}$ take the set $S_i$ and one variable of type $S^l$, which gives $p-1$ variables of type $E_{\nu_i}^l$, $p-1$ variables of type $E_{\bar{\nu_i}}^l$ and one $S^l$ variable. Such a set contains a contraction to an $S^{l+1}$ variable due to Lemma \ref{liftingonespecialvariablehigher}. Thus, one obtains $x$ variables of type $S^{l+1}$.
	\end{proof}
    \begin{Lemma}
		\label{x-y-Lemmaneu}
		Let $S \in \{C, P\}$ and $x$, $y$ and $z$ be non-negative integers with $y+z \ge \left(2-m\right) p-2$ for some $m \in \{0,1,2\}$ and $x-m \ge 0$. Let there be $\left(2p-2\right)y+p^2-3p+1$ variables of type $E^l$ from which at least $\left(p-1\right)y$ variables are of type $E_{\bar \nu}^l$ for any $0\le \nu \le p$. Furthermore, let there be $px+y+z$ variables of type $S^l$. Then one can contract them to $x+y-m$ variables of type $S^{l+1}$ without using $z+mp$ of the variables of type $S^l$.
	\end{Lemma}
	\begin{proof}
		Using Lemma \ref{FS1} to contract $p$ of the variables of type $S^l$ to an $S^{l+1}$ variable can be done $x-m$ times. This leaves $y+z+mp\ge 2p-2$ variables of type $S^{l}$. One can contract $y$ of them together with the variables of type $E^l$ to $y$ variables of type $S^{l+1}$ due to Lemma \ref{neu}. This gives a total of $x+y-m$ variables of type $S^{l+1}$ as claimed, without using $z+mp$ variables of type $S^l$.
	\end{proof}
	\begin{Lemma}
	\label{moreinm_1-I_0^0}
        Let $x$ be a non-negative integer. Let there be at least $px+p^2-3p+3$ variables of type $E_\nu^l$ from which at least $x$ are of type $E_{\nu \mu}^l$ for some $\mu$ and at least $x$ are of type $E_{\nu \bar \mu}^l$. Then one can contract $px$ of them to $x$ variables of type $E_{\bar \nu}^{l+1}$.
	\end{Lemma}
	\begin{proof}
        Divide the $E_\nu^l$ variables in three sets. One contains $x$ variables of type $E_{\nu \mu}^l$, the next one contains $x$ variables of type $E_{\nu \bar \mu}^l$ and the last one contains the remaining variables.
        
        The statement is trivial for $x=0$, thus one can assume that $x \ge 1$.
        Assume now, that the last set contains $z\ge \left( p-2\right) p +1= p^2-2p+1$ variables, and the first two both contain $y \ge 1$ variables. Then there is an $\eta$ such that the last set contains at least $p-1$ variables of type $E_{\nu \eta}^l$ and one can choose one variable in one of the first two sets, which is of type $E_{\nu \bar \eta}^l$. These $p$ variables contract to an $E_{\bar \nu}^{l+1}$ variable due to Lemma \ref{createavariableinm_1-I_0^1}. Then, one can take one variable in the untouched set and put it in the last set, such that the first two sets both contain $y-1$ variables and the last one contains $z-p+2$ variables.
        
        Starting with $z\ge \left(p-2\right)x+p^2-3p+3$ and $y=x$, after following this process $x-1$ times, one still has at least $p^2-2p+1$ variables in the last set left, while the other two each contain one variable. It follows, that one can contract one more variable of type $E_{\bar \nu}^{l+1}$ as described above, giving a total of $x$ variables of type $E_{\bar \nu}^{l+1}$.
	\end{proof}
    \subsection{Inductive Contractions}
    This subsection uses induction to contract sets of variables at some level to variables more than one level higher.
	\begin{Lemma}
		\label{fromitojwithout2p}
		Let $S \in \{C, P\}$ and $i, j \in \N_0$ with $i \le j \le \tau$ as well as $m\in \Z$ with $m \ge -1$.
		Let there be $p^{\tau-i+1}+mp^{\tau-i}-2$ variables of type $S^i$. Then one can contract them to $p^{\tau-j+1}+ m p^{\tau-j}-2$ variables of type $S^j$ and at least $2p-2$ variables of type $S^l$ for all $l \in \{i, \dots, j-1\}$.
	\end{Lemma}
	\begin{proof}
        For $i=j$ the statement is trivial, thus, the cases $i<j\le \tau$ remain.
		Assume for an $l\in \{i, \dots, j-1\}$ that there are $p^{\tau-l+1} + mp^{\tau-l}-2$ variables of type $S^l$ and $2p-2$ variables of type $S^n$ for all $n \in \{i, \dots, l-1\}$. Lemma \ref{FS2} shows that these variables can be contracted to
		\begin{align*}
		\left \lceil \frac{p^{\tau-l+1}+mp^{\tau-l} + 1}{p} \right \rceil -3 = p^{\tau-l} + m p^{\tau-l-1} -2
		\end{align*}
		variables of type $S^{l+1}$. This leaves at least $2p-2$ variables of type $S^l$ unused. The claim follows via induction.
	\end{proof}
	\begin{Lemma}
		\label{fromitoj}
		Let $S \in \{C, P\}$ and $i, j\in \N_0$ with $i \le j \le \tau$ as well as $m \in \Z$ with $m \ge -1$.
		Let there be $p^{\tau-i+1}+mp^{\tau-i}$ variables of type $S^i$ and for all $l \in \{i, \dots, j-1\}$ let there be an $\nu_l$ and $ 2p-2$ variables of type $E_{\nu_l}^l$. Then one can contract them to $p^{\tau-j+1}+ m p^{\tau-j}$ variables of type $S^j$.
	\end{Lemma}
	\begin{proof}
	For $i=j$ the statement is trivial, thus, the cases $i<j\le \tau $ remain.
		Assume for an $l\in \{i, \dots, j-1\}$ there are $p^{\tau-l+1} + mp^{\tau-l}$ variables of type $S^l$ and $2p-2$ variables of type $E_{\nu_l}^l$. Lemma \ref{FS2} shows that there exist
		\begin{align*}
			\left \lceil \frac{p^{\tau-l+1}+mp^{\tau-l} + 3}{p} \right \rceil -3 = p^{\tau-l} + m p^{\tau-l-1} -2
		\end{align*}
		contractions to variables $S^{l+1}$, each of them containing at most $p$ variables. Therefore, there are even $2p$ variables of type $S^{l}$ remaining. Together with the $2p-2$ variables of type $E_{\nu_l}^l$, they can be contracted to another two $S^{l+1}$ variables, using Lemma \ref{2p-1} twice. This gives a total of $p^{\tau-l}+mp^{\tau-l-1}$ variables of type $S^{l+1}$. The claim follows via induction.
	\end{proof}
	\begin{Lemma}
		\label{hinterlassen}
		Let $m \le p-1$ be an integer and let there be a $j \in \{ 0, 1, \dots, \tau-1\}$ such that there are 
		\begin{align*} 
			2 p^{\tau-j+1} +\left(4-2m \right) p^{\tau-j} - \frac{p-1}{2} \sum_{i=1}^{\tau-j-1}p^i + \left( 2m-1\right) p^{\tau-j-1} + 3 \sum_{i=0}^{\tau-j-2} p^i -2p-2,
		\end{align*}
		variables of type $E_\nu^j$. Then one can contract them to $p-m-1$ variables of type $E_\nu^{\tau}$ and $2p-2$ variables of type $E_\nu^i$ for all $i \in \{j, j+1 , \dots, \tau -1\}$.
	\end{Lemma}
	\begin{proof}
        If $j \le \tau-2$, assume that for some $l \in \{j, j+1. \dots, \tau-2\}$ one can contract the variables to $ 2 p^{\tau-l+1}+\left(4-2m \right) p^{\tau-l} - \frac{p-1}{2} \sum_{i=1}^{\tau-l-1}p^i + \left( 2m-1\right) p^{\tau-l-1} + 3 \sum_{i=0}^{\tau-l-2} p^i -2p-2$ variables of type $E_\nu^l$ and $2p-2$ variables of type $E_\nu^i$ for all $i \in \{ j, j+1 , \dots, l-1\}$. Using Lemma \ref{FS6}, the variables of type $E_\nu^l$ can be contracted to
		\begin{align*}
		&\left \lceil \frac{2 p^{\tau-l+1}+\left(4-2m \right)  p^{\tau-l} - \frac{p-1}{2} \sum_{i=1}^{\tau-l-1}p^i + \left( 2m-1\right) p^{\tau-l-1} + 3 \sum_{i=1}^{\tau-l-2} p^i -2p+1}{p} \right \rceil -2p + \frac{p-3}{2} \\
		&= 2 p^{\tau-l} +\left(4-2m \right)  p^{\tau-l-1} - \frac{p-1}{2} \sum_{i=0}^{\tau-l-2}p^i + \left( 2m-1\right) p^{\tau-l-2} + 3 \sum_{i=0}^{\tau-l-3} p^i -1 -2p +\frac{p-3}{2}\\
		&=2 p^{\tau-\left(l+1\right) +1 }+\left(4-2m \right)  p^{\tau-\left(l +1 \right)} - \frac{p-1}{2} \sum_{i=1}^{\tau-\left(l+1\right)-1}p^i + \left( 2m-1\right) p^{\tau-\left(l+1 \right)-1} + 3 \sum_{i=0}^{\tau-\left(l+1\right)-2} p^i -2p-2
		\end{align*}
		variables of type $E_\nu^{l+1}$, while leaving at least $6p-9 \ge 2p-2$ variables of type $E_\nu^l$ unused. Hence, by induction, one can contract the $E_\nu^j$ variables to 
		\begin{align*} 
			2 p^{\tau-\left(\tau-1\right)+1} +\left(4-2m \right)  p^{\tau-\left(\tau-1\right)} - \frac{p-1}{2} \sum_{i=1}^{\tau-\left(\tau-1\right)-1}p^i + \left( 2m-1\right) p^{\tau-\left(\tau-1\right)-1} + 3 			\sum_{i=0}^{\tau-\left(\tau-1\right)-2} p^i -2p-2 \\
			= 2p^2+\left(2-2m \right)  p +2m -3
		\end{align*}
		variables of type $E_\nu^{\tau-1}$ and $2p-2$ variables of type $E_\nu^i$ for all $i \in \{j, j+1, \dots, \tau-2\}$. This reduced the cases $j \le \tau -2$ to the case $j =\tau-1$.
		For $j=\tau-1$, one can contract the variables of type $E_\nu^{\tau-1}$ to
		\begin{align*}
			\left \lceil \frac{ 2p^2 +\left(2-2m \right)  p +2m -3 }{2p-2}\right \rceil  - 4 = p -m-1
		\end{align*}
		variables of type $E_\nu^\tau$ with Lemma \ref{FS6klein}, while leaving at least $6p-9\ge2p-2$ variables of type $E_\nu^{\tau-1}$. This proves the claim.
    \end{proof}
	\begin{Lemma}
		\label{hinterlassen5}
		Let $p=5$ and $m \le p-1$ be an integer. Let there be
		\begin{align*}
			3 p^{\tau-j+1} -m p^{\tau-j} -3 p^{\tau-j} - \sum_{i=0}^{\tau-j-1} p^i -2p +2,
		\end{align*}
		variables of type $E_\nu^j$ for some $j \in \{ 0, 1, \dots, \tau\}$. Then one can contract them to $p-m-1$ variables of type $E_\nu^\tau$ and $2p-2$ variables of type $E_\nu^i$ for all $i \in \{j, j+1, \dots, \tau-1\}$.
	\end{Lemma}
	\begin{proof}
        For $j=\tau$ the claim is trivial, thus, one can assume that $j \in \{0,1,\dots,\tau-1\}$. 
        
        For $j \le \tau-2$, assume that for some $l \in \{j, j+1. \dots, \tau-2\}$ one can contract the variables to $ 3 p^{\tau-l+1} -m p^{\tau-l} -3 p^{\tau-l} - \sum_{i=0}^{\tau-l-1} p^i -2p +2$ variables of type $E_\nu^l$ and $2p-2$ variables of type $E_\nu^i$ for all $i \in \{ j, j+1 , \dots, l-1\}$. Using Lemma \ref{FS6} for $p=5$, the variables of type $E_\nu^l$ can be contracted to
		 \begin{align*}
		 	&\left \lceil \frac{3 p^{\tau-l+1} -m p^{\tau-l} -3 p^{\tau-l} - \sum_{i=0}^{\tau-l-1} p^i -2p +2}{p} \right \rceil -2p + 3\\
			&=3 p^{\tau-l} -m p^{\tau-l-1} -3 p^{\tau-l-1} - \sum_{i=0}^{\tau-l-2} p^i -2 +1-2p+3\\
			&=3 p^{\tau-\left(l+1\right)+1} -m p^{\tau-\left(l+1\right)} -3 p^{\tau-\left(l+1\right)} - \sum_{i=0}^{\tau-\left(l+1\right)-1} p^i -2p +2
		 \end{align*}
		 variables of type $E_{\nu}^{l+1}$, while leaving at least $6p-9 \ge 2p-2$ variables of type $E_\nu^l$ unused. By induction, it follows that one can contract
		 \begin{align*}
		 3 p^{\tau-\left(\tau-1\right)+1} -m p^{\tau-\left(\tau-1\right)} -3 p^{\tau-\left(\tau-1\right)} - \sum_{i=0}^{\tau-\left(\tau-1\right)-1} p^i -2p +2=3p^2-mp-5p+1
		 \end{align*}
		 variables of type $E_\nu^{\tau-1}$ and $2p-2$ variables of type $E_\nu^i$ for all $i \in \{j, j+1, \dots, \tau-2\}$. This reduced the cases $j \le \tau-2$ to the case $j=\tau-1$.
		 
		  For $j =\tau-1$ one has $3p^2-mp-5p+1$ variables of type $E_\nu^j$. This is at least as big as $2p^2-2p+1$ for $m \le 2$. Thus, one can use Lemma \ref{FS6} for $p=5$ to contract them to
        \begin{align*}
		 	\left \lceil \frac{3p^2-mp-5p+1}{p} \right \rceil -2p + 3=3p-m-5+1-2p+3=p-m-1
        \end{align*}
        variables of type $E_\nu^{\tau}$ while leaving at least $2p-2$ variables of type $E_\nu^{\tau-1}$ unused. For $m=4$ the claim follows because $p-4-1=0$, which leaves $3p^2-4p-5p+1=6p+1\ge 2p-2$ variables of type $E_\nu^{\tau-1}$. In the remaining case $m=3$, one obtains
        \begin{align*}
		 	\left \lceil \frac{3p^2-3p-5p+1}{2p-2} \right \rceil -4=1=p-3-1=p-m-1
        \end{align*}
        variables of type $E_\nu^\tau$ with Lemma \ref{FS6klein} while leaving at least $6p-9\ge 2p-2$ variables of type $E_\nu^{\tau-1}$ unused.
	\end{proof}
    \begin{Lemma}
		\label{2ptotau-1}
		Let there be $4 p^{\tau-j} - \frac{p-1}{2} \sum_{i=1}^{\tau-j-1} p^i+ 3 \sum_{i=0}^{\tau-j-2} p^i -2p -2$ variables of type $E_\nu^j$ for some $j \in \{ 0, 1, \dots, \tau-1\}$. Then one can contract $2p-2$ variables of type $E_\nu^i$ for  all $i \in \{j, \dots, \tau-1\}$, simultaneously.
	\end{Lemma}
	\begin{proof}
        For $j=\tau-1$ the statement is trivial, thus, the cases $j \in \{0, 1, \dots, \tau-2\}$ remain.
		Assume that for some $l \in \{j, \dots, \tau-2\}$ one can contract the variables of type $E_\nu^j$ to $2p-2$ variables in $E_\nu^i$ for all $i \in \{j, \dots, l-1\}$ and 
            $4 p^{\tau-l} - \frac{p-1}{2} \sum_{i=1}^{\tau-l-1} p^i+ 3 \sum_{i=0}^{\tau-l-2} p^i -2p -2$
		variables of type $E_\nu^l$. Then they can be contracted with Lemma \ref{FS6} to 
		\begin{align*} 
			 &\left \lceil \frac{4 p^{\tau-l} - \frac{p-1}{2} \sum_{i=1}^{\tau-l-1} p^i+ 3 \sum_{i=0}^{\tau-l-2} p^i -2p -2}{p} \right \rceil  -2p + \frac{p-3}{2} \\
			  &= 4 p^{\tau-l-1} - \frac{p-1}{2} \sum_{i=0}^{\tau-l-2} p^i+ 3 \sum_{i=0}^{\tau-l-3} p^i -2 +1 -2p+\frac{p-3}{2}\\
			&=4 p^{\tau-\left( l+1\right)} - \frac{p-1}{2} \sum_{i=1}^{\tau-\left(l+1\right)-1} p^i+ 3 \sum_{i=0}^{\tau-\left(l+1\right)-2} p^i -2p -2
		\end{align*}
		variables of type $E_\nu^{l+1}$, while leaving at least $6p-9 \ge 2p-2$ variables of type $E_\nu^l$. Via induction one can deduce that on can contract $2p-2$ variables of type $E_\nu^i$ for all $i \in \{ j, \dots, \tau-2\}$ and $ 4 p^{1} - \frac{p-1}{2} \sum_{i=1}^{0} p^i+ 3 \sum_{i=0}^{-1} p^i -2p -2= 2p-2$ variables of type $E_\nu^{\tau-1}$.
	\end{proof}
	\section{Pairs of Forms with $\tau =1$}
	\label{Pairs of forms with tau=1}
	This section contains the proof that for all proper $p$-normalised pairs $f, g$ with $\tau=1$ the equations $f=g=0$ have a non-trivial $p$-adic solution.
	This is primarily done by contracting a $C^{\tau+1}=C^2$ variable if $I_0^0 \ge p^2+p-1$, which indicates that the colour $0$ is zero-representing, and else by contracting a $P^{\tau+1}=P^2$ variable.
	
	The following lemma will exploit $p$-equivalence classes by transforming some pairs $f, g$ into $p$-equivalent pairs $\tilde f, \tilde g$, for which one can contract a $P^2$ variable.
	\begin{Lemma}
		\label{i,i+1}
		Let $1 \le m \le p$ be a natural number and $j \in\{0,\dots, k-1\}$. Let $f, g$ be a pair given by \eqref{f,g} with integer coefficients, $\tau=1$, $q_j \ge pm$, $m_j \ge m \left(2p-1\right)$, $q_{j+1}\ge p-m$ and $I_\mx^{j+1} \ge p-1$. Then there exists a non-trivial $p$-adic solution of $f=g=0$.
	\end{Lemma}
	\begin{proof}
		Apply $x \mapsto px$ for all variables at level $l$ for all $l \in \{0, \dots, j-1\}$, and then multiply both equations with $p^{-j}$. This transforms the pair $f, g$ into a $p$-equivalent pair with integer coefficients, $q_0\ge pm$, $m_0\ge m \left(2p-1\right)$, $q_1 \ge p-m$ and $I_\nu^1=I_\mx^1 \ge p-1$ for some $\nu$. Using Lemma \ref{5.1}, one can contract the $E^0$ variables to $m$ variables of type $P^1$. The $p-1$ variables of type $E_\nu^1$ and the $p-m$ variables of type $E_{\bar \nu}^1$ can be contracted together with the $P^1$ variables to a $P^2$ variable due to Lemma \ref{liftingonespecialvariablehigher}. Thus, the transformed pair has a non-trivial $p$-adic solution, from which it follows that the $p$-equivalent pair $f, g$ has one as well.
	\end{proof}
	Due to this lemma, one can assume that if $q_j \ge pm$ and $m_j \ge m \left(2p-1\right)$ for some $j \in \{0, \dots, k-1\}$ that either $q_{j+1} \le p-m-1$ or $I_\mx^{j+1} \le p-2$. For a $p$-normalised pair, one has $m_0 \ge 2p^2-2p+1 \ge \left(p-1\right)\left(2p-1\right)$ and $q_0 \ge p^2-p+1 \ge \left(p-1\right) p$. Therefore, one can assume that one has either $q_1 =0$ or $I_\mx^1 \le p-2$. The following two lemmata will divide the case $\tau =1$ into $I_\mx^1 \ge p-1$ and thus $q_1=0$ and $I_\mx^1 \le p-2$.
	\begin{Lemma}
        Let $f, g$ be a proper $p$-normalised pair with $\tau =1$ and $I_\nu^1=I_\mx^1 \ge p-1$. Then the equations $f=g=0$ have a non-trivial $p$-adic solution.
	\end{Lemma}
    \begin{proof}
		As describe above, one can assume that $q_1 =0$ and thus $I_\nu^1=I_\mx^1=m_1$. It follows that
		\begin{align}
		\label{low}
            m_0 \ge 3p^2-3p+1-q_1 =3p^2-3p+1 \ge 2p^2-p.
		\end{align}
		Assume first, that $r\left(f, g\right)=r \ge 0$. Then one can use Lemma \ref{5.1} to contract the $E^0$ variables to $p$ variables of type $P^1$ and Lemma \ref{2p-1} to contract the $P^1$ variables together with the $E_\nu^1$ variables to a $P^2$ variable. 
		Consequently, one can assume, that $r=-1$ which leads to
		\begin{align*}
            I_0^0 \ge 3p^2-3p+1-q_0-q_1 \ge 2p^2-3p+2 \ge p^2 +p-1.
		\end{align*}
        Hence, the colour $0$ is zero-representing and it suffices to show that one can contract a $C^2$ variable.
        
		By Lemma \ref{properties}, one knows that $\nu \in \{0,p\}$. If $\nu=p$, one can contract $2p-2$ of the variables of type $E_0^0$ to an $E_0^1$ variable, using Lemma \ref{FS5} once, because $2 p^2 -3p+2\ge 8p-7=2\left(4p-4\right)+1$ for all $p\ge 5$. If on the other hand $\nu=0$, one has
		\begin{align*}
		I_{00}^{0} \ge  \frac{I_0^0}{p} \ge 2p-3\ge p-1,
		\end{align*}
		due to Lemma \ref{properties} and 
		\begin{align*}
		I_0^0 - I_{00}^0 \ge 2p^2-2p-q_0-\left(m_1-I_0^1\right) \ge p^2 -2p+1\ge 1,
		\end{align*}
		by \eqref{NichtI_00Grenze}.
        Hence, one can contract $p-1$ variables of type $E_{00}^1$ and one $E_{0\bar 0}^0$ variable to an $E_{\bar 0}^1$ variable due to Lemma \ref{createavariableinm_1-I_0^1}.
        
        In both cases, there are still at least $2p^2-3p+2-\left(2p-2\right) =2p^2-5p+4\ge 2p-2$ variables of type $E_0^0$ remaining. Those contract with $p^2-p$ of the $C^0$ variables to $p-1$ variables of type $C^1$ due to Lemma \ref{fromitoj}. All in all, one has $p-1$ variables of type $E_\nu^1$, one $E_{\bar \nu}^1$ variable and $p-1$ variables of type $C^1$. Due to Lemma \ref{liftingonespecialvariablehigher} these can be contracted to a $C^2$ variable, which completes the proof.
    \end{proof}
	\begin{Lemma}
        Let $f, g$ be a proper $p$-normalised pair with $\tau =1$ and $I_\nu^1=I_\mx^1 \le p-2$. Then the equations $f=g=0$ have a non-trivial $p$-adic solution.
	\end{Lemma}
	\begin{proof}
	    By $I_\mu^1 \le I_\mx^1\le p-2$ for all  $0\le \mu \le p$, it follows that
		\begin{align*}
		m_1 \le \left( p-2\right) \left( p+1\right) = p^2-p-2.
		\end{align*}
		If one has $q_1\ge p$ and $m_1 \ge 2p-1$ one can assume, due to Lemma \ref{i,i+1}, that either $q_2\le p-2$ or $I_\mx^2 \le p-2$.	For $q_2 \le p-2$ it follows that
		\begin{align*}
		m_0 \ge 5p^2 -5p+1 -p^2 +p+2 -p +2 = 4 p^2-5p +5 \ge 4p^2-6p+3,
		\end{align*}
		while for $I_\mx^2 \le p-2$ it follows that $m_2 \le p^{2}-p-2$ and thus
		\begin{align*}
		m_0 \ge 6p^2-6p+1 - p^2+p+2 -p^2 + p+2 = 4 p^2 -4 p + 5 \ge 4p^2-6p+3.
		\end{align*}
		Else, one has either $q_1 \le p-1$ or $m_1 \le 2p-2$. If $q_1 \le p-1$ it follows that $m_1 \le 2p-2$ as well, because $m_1=I_\mx^1 + q_1$. Then one obtains
		\begin{align*}
		m_0 \ge 4p^2 - 4 p +1 - 2p+2 = 4p^2 -6p +3.
		\end{align*}
		One of these three bounds holds in any case, thus, one can assume that 
		\begin{align}
				\label{m_0tau=1}
		m_0 \ge 4p^2 -6p +3.
		\end{align}
		This lower bound for $m_0$ leads to
		\begin{align}
		\label{I_0^0tau=1}
		I_0^0 =m_0-q_0 \ge 4 p^2 -6 p+ 3 -p^2 -\left(r+1\right) p +1 = 3p^2 -rp-7p +4.
		\end{align}
		For $r \le p-2$ this is at least as big as $p^2+p-1$ for $p\ge 5$, hence, it suffices to contract a $C^2$ variable, whereas one has to contract a $P^2$ variable for $r=p-1$. The remaining proof will be divided into three cases, based on the value of $r=r\left(f, g\right)$.
		\paragraph{\textbf{Case} $\mathbf{r=p-1}$}
		If $m_0 \ge \left(2p-1\right) \left(2p-1\right) = 4p^2 -4p+1$, one can use Lemma \ref{5.1} to contract the $E^0$ variables to $2p-1$ variables of type $P^1$. By Lemma \ref{p^n=2}, it follows that one can contract those $P^1$ variables to a $P^2$ variable. Hence, one can assume that $m_0 \le 4p^2-4p$ and thus $m_1 \ge 1$. Due to \eqref{m_0tau=1} one has $m_0\ge 4 p^2 - 6p +2=\left(2p-1\right) \left(2p-2\right) $. Therefore, Lemma \ref{5.1} shows, that one can contract the $E^0$ variables to $2p-2$ variables of type $P^1$. Lemma \ref{p^n} with $n=2$ shows that one can contract them together with one of the $E^1$ variables to a variable of a level at least $2$. This contraction cannot contain only the $E^1$ variable, thus the resulting variable has to be a $P^2$ variable. 

        \paragraph{\textbf{Case} $\mathbf{0 \le r \le p-2}$}
        One can assume that $I_\nu^1=I_\mx^1 \le p-r-2$, because else, Lemma \ref{fromitoj} can be used to contract $p^2+rp$ of the $C^0$ variables together with $2p-2$ variables of type $E_0^0$ to $p+r$ variables of type $C^1$. Then one can contract them together with the $E_\nu^1$ variables to a $C^2$ variable due to Lemma \ref{2p-1}.
        It follows that that
		\begin{align}
		\label{x}
		m_1 \le p^2 -\left(r+1\right)p -r-2.
		\end{align}
		If $q_2\ge p-1$ and $I_\mx^2\ge p-1$, one can use Lemma \ref{contractiontolevell+1} to contract $p\left(p-r-1\right)$ of the variables of type $E_0^0$ to $p-r-1$ variables of type $E^1$. This is possible, because afterwards, there are still at least 
		\begin{align*}
		3p^2-rp-7p+4-p\left(p-r-1\right) =2p^2-6p+4 \ge 3p-2
		\end{align*}
		of the $E_0^0$ variables unused. Lemma \ref{fromitoj} can be used to contract $p^2+rp$ of the $C^0$ and $2p-2$ of the remaining $E_0^0$ variables to $p+r$ variables of type $C^1$. One can assume, that the $C^1$ variables are of type $E^1$, because else one already has a $C^2$ variable. Take the set $\mathscr{K}$ of the $2p-1$ variables of type $E^1$ that were contracted, from which $p+r$ are of type $C^1$. If there is a $\mu$ with $I_\mu\left(\mathscr{K}\right) \ge p$ there are at least $p$ variables of type $E_\mu^1$ in $\mathscr{K}$. As $p+r$ of the variables in $\mathscr{K}$ are type $C^1$, it follows that there is at least one $C_\mu^1$ variable in $\mathscr{K}$. Thus, one can contract the $E_\mu^1$ variables in $\mathscr{K}$ with Lemma \ref{onecolourfulandp-1differentvariablesofthesamecolour} to a $C^2$ variable. Else, one has $q\left(\mathscr{K}\right) \ge p$ and thus, one has transformed the pair $f, g$ into another one with $m_1 \ge 2p-1$ and $q_1 \ge p$. This new pair has the same values for $q_2$ and $I_\mx^2$, thus it follows from Lemma \ref{i,i+1} that it has a non-trivial $p$-adic solution. Consequently the pair $f, g$ has one as well. Thus, one can assume, that either $q_2 \le p-2$ or $I_\mx^2 \le p-2$. 
		
		By \eqref{x}, it follows for $q_2 \le p-2$ that
		\begin{align*}
		m_0 \ge 5p^2-5p+1 -p^2 +\left(r+1\right)p +r+2 -p+2= 4p^2-5p+rp+5+r
		\end{align*}
		and for $I_\mx^2 \le p-2$ that $m_2 \le p^2-p-2$ and therefore
		\begin{align*}
		m_0 \ge 6p^2-6p+1 -p^2 +\left(r+1\right)p+r+2-p^2+p+2= 4p^2-4p +rp +5+r.
		\end{align*}
		In both cases, one obtains the lower bound
		\begin{align*}
		m_0 \ge 4p^2-5p+rp+5+r,
		\end{align*}
		which leads to 
		\begin{align*}
		I_0^0=m_0-q_0 \ge 3p^2-6p+6+r\ge 2p^2-2rp+2r-1.
		\end{align*}
		Now one can distinguish between the cases $m_1 \ge 1$ and $m_1 =0$.
		
		\subparagraph{\textit{Case} $m_1 \ge 1$} One can use Lemma \ref{FS6klein} to contract the $E_0^0$ variables to
		\begin{align*}
		\left \lceil \frac{  2p^2-2rp+2r-1}{2p-2} \right \rceil -4= p-r-2
		\end{align*}
		variables of type $E_0^1$. This leaves at least $6p-9 \ge 2p-2$ variables of type $E_0^0$. Hence, one can use Lemma \ref{fromitoj} to contract them with $p^2+rp$ of the $C^0$ variables to $p+r$ variables of type $C^1$. The set $\mathscr{H}$, containing the $p-r-2$ variables of type $E_0^1$, the $p+r$ variables of type $C^1$ and one further $E^1$ variables, which exists due to $m_1\ge1$, contains a contraction to a $C^2$ variable. If none of the $C^1$ variables is already of type $C^2$, there is either a $\mu$ such that $I_\mu\left(\mathscr{H}\right) \ge p$ or $q\left( \mathscr{H} \right)\ge p$. If $I_\mu\left(\mathscr{H}\right) \ge p$, then at least one of the $E_\mu^1$ variables in $\mathscr{H}$ is a $C^1$ variable and thus $\mathscr{H}$ contains a contraction to a $C^2$ variable due to Lemma \ref{onecolourfulandp-1differentvariablesofthesamecolour}. If on the other hand $q\left(\mathscr{H}\right) \ge p$, then $\mathscr{H}$ contains a contraction to a variable at level at least $2$, which can be traced back to at least two variables of different colour at level $1$, due to Lemma \ref{TAEL3}.
		The only way that such a variable is not of type $C^2$, is that the contraction contains no $C^1$ variable. The variables in $\mathscr{H}$ which are not of type $C^1$ are $p-r-2$ variables of type $E_0^1$ and one $E^1$ variable. As the contracted variable can be traced back to two variables of different colours at level $1$, the $E^1$ variable has to be an $E_{\bar 0}^1$ variable. But if a subset $\mathscr{K}$ of $\mathscr{H}$ contains this variable and additionally only variables of type $E_0^1$, then it cannot be a contraction to a variable at level at least $2$, because then one has exactly one $i \in \mathscr{K}$ for which the second entry $\tilde b_i$ of the level coefficient vector is not congruent to $0$ modulo $p$. Therefore, one cannot solve $\sum_{j \in \mathscr{K}} \tilde b_j y_j^k \equiv 0 \bmod p$ with all $y_j \nequiv 0 \bmod p$. Consequently, this cannot occur, and the resulting variable is a $C^2$ variable.
		\subparagraph{\textit{Case} $m_1 =0$}This leads to the even better bound
		\begin{align*}
		m_0 \ge 4p^2-4p+1
		\end{align*} 
		and thus
		\begin{align*}
		I_0^0 \ge 4p^2-4p+1-p^2-\left(r+1\right)p +1= 3p^2 -\left(5+r\right)p +2.
		\end{align*}
		For $p\ge7$, this is at least as big as $2p^2+2p-2rp+2r-3$, thus, one can use Lemma \ref{hinterlassen} to contract the $E_0^0$ variables to $p-r-1$ variables of type $E_0^1$, while leaving at least $2p-2$ variables of type $E_0^0$ unused. For $p=5$, this is at least as big as $3p^2-rp-5p+1$, thus Lemma \ref{hinterlassen5} shows that one can contract the $E_0^0$ variables to $p-r-1$ variables of type $E_0^1$ as well, while leaving at least $2p-2$ variables of type $E_0^0$ unused. In both cases, one can use Lemma \ref{fromitoj} to contract the $2p-2$ variables of type $E_0^0$ with $p^2+rp$ of the $C^0$ variables to $p+r$ variables of type $C^1$. Then one can contract them together with the $p-r-1$ variables of type $E_0^1$ to a $C^2$ variable due to Lemma \ref{2p-1}.
		\paragraph{\textbf{Case} $\mathbf{r=-1}$}
		Note first, that one has $m_1-I_0^1 \le p^2-2p =\left(p-2\right)p$ due to $I_\mx^1\le p-2$, and thus
		\begin{align*}
		I_0^0-I_{00}^0 \ge 2p^2-2p-q_0 -\left(m_1-I_0^1\right)\ge1,
		\end{align*}
		by \eqref{NichtI_00Grenze}.
		If $m_0 \ge 4p^2-4p$, one obtains the lower bound
		\begin{align*}
		I_0^0 \ge 3p^2-4p+1,
		\end{align*}
		and consequently
		\begin{align*}
		I_{00}^0\ge \frac{I_0^0}{p} \ge 3p-4 \ge p-1.
		\end{align*}
		Therefore, one can take $p-1$ variables of type $E_{00}^0$ and one of type $E_{0\bar 0}^0$, which can be contracted to a $E_{\bar 0}^1$ variable by Lemma~\ref{createavariableinm_1-I_0^1}. There are at least $3p^2-5p+1$ variables of type $E_0^0$ remaining, which can be contracted to $p-1$ variables of type $E_0^1$ using Lemma~\ref{FS6klein} for $p\ge7$ and Lemma~\ref{FS6} for $p=5$. This leaves at least $6p-9 \ge 2p-2$ variables of type $E_0^0$, which can be contracted with $p^2-p$ of the $C^0$ variables to $p-1$ variables of type $C^1$ using Lemma \ref{fromitoj}. Then one can use Lemma \ref{liftingonespecialvariablehigher} to contract the $p-1$ variables of type $E_0^1$, the $p-1$ variables of type $C^1$ and the $E_{\bar 0}^1$ variable to a $C^2$ variable. Hence, one can assume that 
		\begin{align*}
		m_0 \le 4p^2-4p-1.
		\end{align*}
		It follows that $m_1 \ge 2$. Note, that one has
		\begin{align*}
            I_0^0 \ge 3p^2 -6p+4 \ge 2p^2-1=\left(2p-2\right)\left(p+1\right)+1 \quad \textup{and} \quad I_{00}^0 \ge 3p-6 \ge p-1
		\end{align*}
		due to \eqref{I_0^0tau=1}.
		\subparagraph{\textit{Case} $m_1-I_0^1=0$}
		Due to $m_1 \ge 2$, one has $I_0^1 \ge 2$. Take a set, which contains $p-1$ variables of type $E_{00}^0$ and one $E_{0\bar0}^0$ variable. This set contains a contraction to an $E_{\bar 0}^1$ variable due to Lemma \ref{createavariableinm_1-I_0^1}. Then there are at least $3p^2-7p+4\ge 2p^2-2p+1$ variables of type $E_0^0$ left. Therefore, one can use Lemma \ref{FS6klein} to contract them to $p-3$ variables of type $E_0^1$, giving a total of $p-1$, while leaving at least $6p-9\ge 2p-2$ variables of type $E_0^0$ unused. Lemma \ref{fromitoj} can be used to contract $2p-2$ of the remaining $E_0^0$ variables together with $p^2-p$ of the $C^0$ variables to $p-1$ variables of type $C^1$. One can contract the $p-1$ variables of type $E_0^1$, the $E_{\bar 0}^1$ variable and the $p-1$ variables of type $C^1$ to a $C^2$ variable, due to Lemma \ref{liftingonespecialvariablehigher}.
		\subparagraph{\textit{Case} $m_1-I_0^1 \ge 1$}  Use Lemma \ref{FS6klein} to contract the $E_0^0$ variable to $p-2$ variables of type $E_0^1$ while leaving at least $6p-9 \ge 2p-2$ unused. Then one can take Lemma \ref{fromitoj} to contract $p^2-p$ of the $C^0$ variables together with $2p-2$ of the remaining $E_0^0$ variables to $p-1$ variables of type $C^1$. If $I_0^1 \ge 1$, then one can use Lemma \ref{liftingonespecialvariablehigher} to contract the $p-1$ variables of type $E_0^1$, the $p-1$ variables of type $C^1$ and one of the $E_{\bar 0}^1$ variables to a $C^2$ variable. Thus, one can assume, that $I_0^1=0$, $m_1-I_0^1\ge2$ and
		\begin{align*}
            m_1 \le p^2-2p,
		\end{align*}
		because $I_\mx^1 \le p-2$.
        If none of the $C^1$ variable is already of type $C^2$, they are all $E^1$ variables. Take a set $\mathscr{K}$ containing the $C^1$ variables, two of the $E_{\bar 0}^1$ variables which exist due to $m_1-I_0^1 \ge 2$ and the $p-2$ variables of type $E_0^1$. If there is a $\mu$ such that $I_\mu\left(\mathscr{K} \right) \ge p$, then there is at least one $C_\mu^1$ variable in $\mathscr{K}$. Due to Lemma \ref{onecolourfulandp-1differentvariablesofthesamecolour} one can contract the variables in $\mathscr{K}$ of colour $\mu$ to an $C^2$ variable. Else, one has $q \left( \mathscr{K}\right) \ge p$, because $|\mathscr{K}| = 2p-1$. It follows, that one has transformed the pair $f, g$ into a pair with $m_1 \ge 2p-1$ and $q_1 \ge p$. The new pair either has a non-trivial $p$-adic solution due to Lemma \ref{i,i+1}, from which it would follow that $f, g$ has one as well, or it has $q_2 \le p-2$ or $I_\mx^2 \le p-2$. As the new pair has the same parameter $q_2$ and $I_\mx^2$ as the pair $f, g$, one can assume, that $q_2 \le p-2$ or $I_\mx^2\le p-2$ holds for $f, g$ as well.
		This contradicts the $p$-normalisation, because then one of the inequalities
		\begin{align*}
            \quad m_0 + m_1 + q_2 \le 4p^2-4p-1 + p^2-2p +p-2= 5p^2-5p-3 < 5p^2-5p+1,
		\end{align*}
		and
		\begin{align*}
		  \qquad \quad \enskip m_0 + m_1 + m_2 \le 4p^2 -4p -1 + p^2-2p+ p^2-p-2 = 6p^2-7p-3 < 6p^2-6p+1,
		\end{align*}
        holds, hence, it follows that this case cannot occur.

    This concludes the case $r=-1$ and with that the claim follows.
	\end{proof}
    This shows that for every proper $p$-normalised pair $f,g$ the equations $f=g=0$ have a non-trivial $p$-adic solution provided that $\tau=1$.
    \section{Pairs of Forms with $\tau\ge2$}
    \label{Pairs of forms with tau ge 2}
    This section will prove the theorem for $\tau \ge2$, which completes the proof. In general, the proof relies on the same techniques independent on the actual value of $\tau$, but sometimes one has to separate the cases $\tau=2$ and $\tau=3$, because the proof is easier for bigger $\tau$ and hence, the cases $\tau \in \{2,3\}$ require some extra effort. 
    
    In order to avoid a repetition of the same argument, the following lemma will point out a situation in which one can contract a $C^{\tau+1}$ or a $P^{\tau+1}$ variable, which will appear constantly in the proof for $\tau \ge2$.
    \begin{Lemma}
		\label{ziel}
		Let $S \in \{C, P\}$ and $0 \le m \le p-1$. Let there be $p^{\tau-j+1}+ mp^{\tau-j}$ variables of type $S^j$ for some $j \in \{ 0, \dots, \tau -1\}$ and $p-m-1$ variables of type $E_\nu^\tau$ for some $\nu$. Furthermore for $i \in \{ j, j+1, \dots, \tau-1\}$ let there be $2p-2$ variables of type $E_{\nu_i}^i$ for some colours $\nu_i$. Then one can contract them to a variable of type $S^{\tau+1}$.
	\end{Lemma}
	\begin{proof}
		One can contract the variables of type $S^j$ and type $E_{\nu_i}^i$ for $i \in \{j , \dots, \tau-1\}$ to $p+m$ variables at level of type $S^{\tau}$ due to Lemma \ref{fromitoj}. Those and the $p-m-1$ variables of type $E_\nu^\tau$ can be contracted to a variable of type $S^{\tau+1}$ using Lemma \ref{2p-1}.
	\end{proof}
    The following lemma focuses on cases, where the number of variables at level $0$ is small.
    \begin{Lemma}
        Let $f, g$ be a proper $p$-normalised pair with $\tau \ge 2$ and $m_0 \le 3p^{\tau+1}-4p^{\tau} - 2p^{\tau-1} +p+3$. Then the equations $f=g=0$ have a non-trivial $p$-adic solution.
    \end{Lemma}
    \begin{proof}
        By the $p$-normalisation of $f, g$, one has $q_0 \ge p^{\tau+1} - p^\tau+1$ and $m_0 \ge 2p^{\tau+1} -2p^{\tau}+1$, from which it follows that one can contract the variables at level $0$ to $p^\tau-p^{\tau-1}$ variables of type $P^1$ due to Lemma \ref{5.1}.
        The upper bound of $m_0$ provides the bounds
        \begin{align*}
            m_1 &\ge 4p^{\tau+1}-4p^\tau+1 - 3p^{\tau+1} + 4p^\tau +2p^{\tau-1}-p-3\\
            &= p^{\tau+1} +2p^{\tau-1} -p-2\ge 2p^{\tau} + 4p^{\tau-1} + p^2 -p-7\\
            &= \left(p^{\tau-1} + 3 \sum_{i=0}^{\tau-2} p^i -1\right) \left(2p-2\right)+ p^2-5p-3
        \end{align*}
        and
        \begin{align*}
            q_1 &\ge 3p^{\tau+1}-3p^\tau+1 - 3p^{\tau+1} + 4p^\tau +2p^{\tau-1}-p-3\\
            &=p^{\tau} +2p^{\tau-1} -p -2=\left(p^{\tau-1} + 3 \sum_{i=0}^{\tau-2} p^i -1\right) \left(p-1\right).
        \end{align*}
        Therefore, there are at least $\left(p^{\tau-1} + 3 \sum_{i=0}^{\tau-2} p^i -1\right) \left(2p-2\right)+ p^2-3p+1$ variables of type $E^1$ from which at least $\left(p^{\tau-1} + 3 \sum_{i=0}^{\tau-2} p^i -1\right) \left(p-1\right)$ are of type $E_{\bar \nu}^1$ for all $0 \le \nu \le p$.
        Those variables can be contracted together with the $P^1$ variables to $2p^{\tau-1}+p^{\tau-2}-2$ variables of type $P^2$ by using Lemma \ref{x-y-Lemmaneu} with $x=p^{\tau-1}-2p^{\tau-2} -3\sum_{i=0}^{\tau-3} p^i -1$, $y=p^{\tau-1} + 3 \sum_{i=0}^{\tau-2} p^i -1$ and $z=p-2$. Then Lemma \ref{fromitojwithout2p} can be used to contract the $P^2$ variables to $2p-1$ variables of type $P^\tau$, which contract to a $P^{\tau+1}$ variable due to Lemma \ref{p^n=2}.
    \end{proof}
    For bigger $m_0$ it will be helpful to divide the cases depending on the value of $r\left(f, g\right)$. The following three lemmata will complete the proof that a for a proper $p$-normalised pair $f, g$ with $\tau \ge2$ and $r=r\left(f,g\right) \ge 0$ the equations $f=g=0$ have a non-trivial $p$-adic solution.
    
    This will be done by using different strategies depending on the size of $m_0$. The area of the value of $m_0$ in which one has to use a certain strategy differs between $p\ge7$ and $p=5$. This is due to some inequalities, which do not hold if $p$ is too small. To counter this, the lemmata that are stronger in the case $p=5$ will be used, which results in the different areas.
    \begin{Lemma}
	    Let $f, g$ be a proper $p$-normalised pair with $\tau \ge 2$, $r =r\left(f, g\right)\ge 0$ and $m_0 \ge 3p^{\tau+1} + 8 p^{\tau}$ for $p \ge 7$ and $m_0 \ge 3p^{\tau+1} +3p^{\tau}$ for $p =5$. Then the equations $f=g=0$ have a non-trivial $p$-adic solution.
    \end{Lemma}
	\begin{proof}
		As $I^0_0 = m_0-q_0$, one can estimate $I_0^0$ via
		\begin{align*}
		I_0^0 &= m_0-q_0\ge  3 p^{\tau +1}+ 8 p^\tau - p^{\tau+1} -\left( r+1\right)  p^\tau+1= 2 p^{\tau+1} +\left( 7-r \right)  p^{\tau}+1,
		\end{align*}
		for all primes $p \ge 7$, and via
		\begin{align*}
			I_0^0 &= m_0-q_0\ge3 p^{\tau +1}+ 3 p^\tau - p^{\tau+1} -\left( r+1\right)  p^\tau+1= 2 p^{\tau+1} +\left( 2-r \right)  p^{\tau}+1,
		\end{align*}
		for $p=5$.
		Both are at least as big as $p^{\tau+1}+ p^\tau -1$, because $r \le p-1$, from which it follows that the colour $0$ is zero-representing, and hence, it suffices to contract a $C^{\tau+1}$ variable. 
		Furthermore, the lower bound for $I_0^0$ implies that
		\begin{align*}
			I_0^0 \ge 2 p^{\tau+1} +\left(4-2r\right) p^{\tau}+ \left(2r-1\right) p^{\tau-1}  + 3 \sum_{i=0}^{\tau-2} p^i -2p-2
		\end{align*}
		for $p\ge 7$ and
		\begin{align*}
            I_0^0 \ge 3 p^{\tau+1} -r p^{\tau} -3 p^{\tau} - \sum_{i=0}^{\tau-1} p^i -2p +2
		\end{align*}
		for $p=5$.
		Thus, one can contract the $E_0^0$ variables to $p-r-1$ variables of type $E_0^\tau$, using Lemma \ref{hinterlassen} for $p \ge7$ and Lemma \ref{hinterlassen5} for $p=5$, while leaving at least $2p-2$ variables of type $E_0^i$ for all $i \in \{0, 1, \dots, \tau-1\}$. Then one can contract $p^{\tau+1} +r p^\tau$ variables of type $C^0$ together with the $2p-2$ variables of type $E_0^i$ for all $i \in \{0, 1, \dots, \tau-1\}$ and the $E_0^\tau$ variables to a $C^{\tau +1}$ variable due to Lemma \ref{ziel}.
	\end{proof}

	\begin{Lemma}
		Let $f, g$ be a proper $p$-normalised pair $f, g$ with $\tau \ge 2$, $r=r \left(f,g\right)\ge 0$, and $m_0 \ge 3p^{\tau+1}+p^\tau-3$ which has $m_0\le 3 p^{\tau +1}+ 8p^\tau-1$ for $p \ge 7$ and $m_0 \le 3p^{\tau+1} +3p^\tau-1$ for $p=5$. Then the equations $f=g=0$ have a non-trivial $p$-adic solution.
	\end{Lemma}
	\begin{proof}
        By $q_0\le 2p^{\tau+1}-2$, one obtains
        \begin{align*}
            I_0^0=m_0-q_0 \ge 3p^{\tau+1}+p^\tau-3-2p^{\tau+1}+2 = p^{\tau+1} + p^\tau-1,
        \end{align*}
        from which it follows that the colour $0$ is zero-representing. Therefore, it suffices to contract a $C^{\tau+1}$ variable. The variables of type $E_0^0$ can be contracted with Lemma \ref{2ptotau-1} to $2p-2$ variables of type $E_0^i$ for all $i \in \{ 0, 1, \dots, \tau-1\}$ as
		\begin{align*}
            I_0^0 \ge p^{\tau+1} +p^\tau -1 \ge 4 p^{\tau} -\frac{p-1}{2}\sum_{i=1}^{\tau-1} p^i + 3 \sum_{i=0}^{\tau-2} p^i -2 p -2.
		\end{align*}
		If $I_\nu^\tau \ge p-r-1$ for some $\nu$, then one can contract the $p^{\tau+1} + rp^{\tau}$ variables of type $C^0$ together with the $p-r-1$ variables of type $E_\nu^\tau$ and the $2p-2$ variables of type $E_0^i$ for all $i \in \{0, \dots, \tau-1\}$ to one variable of type $C^{\tau+1}$ with Lemma \ref{ziel}. Thus one can assume that
		\begin{align}
		\label{Z}
			m_\tau \le \left( p-r-2\right) \left( p+1\right) = p^2 -\left( r+1\right) p -r-2\le p^2.
		\end{align}
		Likewise, if $I_\nu^j \ge2 p^{\tau-j+1} +\left(4-2r\right) p^{\tau-j} - \frac{p-1}{2} \sum_{i=1}^{\tau-j-1}p^i + \left( 2r-1\right) p^{\tau-j-1} + 3 \sum_{i=0}^{\tau-j-2} p^i -2p-2$ for some $j \in \{1, \dots, \tau-1\}$ and some $\nu$, one can contract the variables of type $E_\nu^j$ to $p-r-1$ variables of type $E_\nu^\tau$ due to Lemma \ref{hinterlassen}. Then again one can contract the $p^{\tau+1} + rp^{\tau}$ variables of type $C^0$ together with the $p-r-1$ variables of type $E_\nu^\tau$ and the $2p-2$ variables of type $E_0^i$ for all $i \in \{0, \dots, \tau-1\}$ to a $C^{\tau+1}$ variable with Lemma \ref{ziel}. Hence, one can assume that this is not the case, giving the upper bound
		\begin{align}
		\label{upperbound}
            I_\mx^j \le2 p^{\tau-j+1} +\left(4-2r\right) p^{\tau-j} - \frac{p-1}{2} \sum_{i=1}^{\tau-j-1}p^i + \left( 2r-1\right) p^{\tau-j-1} + 3 \sum_{i=0}^{\tau-j-2} p^i -2p-3.
		\end{align}
		If $m_j\ge 2p^{\tau-j+1}- \left( 2r+2\right) p^{\tau-j} +p^2 -3p+2r+1$ and $q_j \ge p^{\tau-j+1}- \left( r+1\right) p^{\tau-j} +r$ for some $j \in \{1, \dots, \tau-1\}$, one can contract $p^{\tau+1}+rp^\tau$ of the $C^0$ variables together with the $2p-2$ variables of type $E_0^i$ for $i \in \{0, \dots, j-1\}$ to $p^{\tau-j+1} + rp^{\tau-j}$ variables of type $C^j$, using Lemma~\ref{fromitoj}. It follows from the lower bounds for $m_j$ and $q_j$, that one can contract the variables of type $E^j$ together with the $p^{\tau-j+1}+rp^{\tau-j}$ variables of type $C^j$ to $2p^{\tau-j}-p^{\tau-j-1}-1$ variables of type $C^{j+1}$, using Lemma \ref{x-y-Lemmaneu} with $x= p^{\tau-j}-p^{\tau-j-1} +r\sum_{i= 0}^{\tau-j-1} p^i$, $y= p^{\tau-j} -r \sum_{i= 0}^{\tau-j-1} p^i$ and $z=r$. This leaves at least $p+r$ of the $C^j$ variables unused. Furthermore, the $2p-2$ variables of type $E_0^j$ which were contracted at the beginning of the proof are unused as well. Hence, Lemma~\ref{2p-1} can be used to contract $p-1$ of them and $p$ of the remaining $C^j$ variables to another $C^{j+1}$ variable. All in all, one has $2p^{\tau-j}-p^{\tau-j-1}$ variables of type $C^{j+1}$ and $2p-2$ variables of type $E_0^i$ for all $i \in \{j+1, \dots, \tau-1\}$ left. By Lemma \ref{ziel}, these variables contract to a $C^{\tau+1}$ variable. One can therefore assume, that either $m_j \le 2p^{\tau-j+1}- \left( 2r+2\right) p^{\tau-j} +p^2 -3p+2r$ or $q_j\le p^{\tau-j+1}- \left( r+1\right) p^{\tau-j} +r-1$ for $j \in \{ 1, \dots, \tau-1\}$.
		It follows that either one has $m_j \le 2p^{\tau-j+1}-\left( 2r+2\right) p^{\tau-j} +p^2-3p+2r$ or for $q_j \le p^{\tau-j+1}-\left( r+1\right) p^{\tau-j} +r-1$ one obtains, due to \eqref{upperbound}, the upper bound
		\begin{align*}
            m_j \le3 p^{\tau-j+1} +\left(3-3r\right) p^{\tau-j} - \frac{p-1}{2} \sum_{i=1}^{\tau-j-1}p^i + \left( 2r-1\right) p^{\tau-j-1} + 3 \sum_{i=0}^{\tau-j-2} p^i -2p+r-4.
		\end{align*}
		Both upper bounds are smaller than $4p^{\tau-j+1}$, thus one can assume, that $m_j\le 4p^{\tau-j+1}$ for $j \in \{1, \dots, \tau-1\}$. It follows that one has $m_1 \le 4p^\tau$ for all $\tau\ge2$ and $m_2 \le 4p^{\tau-1}\le p^\tau$ for $\tau \ge 3$. Furthermore, one has $m_2 \le p^\tau$ for $\tau=2$ due to \eqref{Z}. It follows that
        \begin{align*}
            m_0 + m_1 + m_2 \le 3 p^{\tau +1}+ 13p^\tau -1\le 6p^{\tau+1}-6p^{\tau}
        \end{align*}
        for all $p \ge 7$, whereas one obtains
        \begin{align*}
            m_0 + m_1 + m_2 \le 3p^{\tau+1} +8p^\tau -1\le 6p^{\tau+1} -6p^\tau
        \end{align*}
        for $p=5$. This contradicts the $p$-normalisation of $f, g$, from which the claim follows.
        \end{proof}
        
        \begin{Lemma}
            Let $f,g$ be a proper $p$-normalised pair with $\tau \ge 2$, $r =r\left(f,g\right)\ge 0$ and $3p^{\tau+1}-4p^\tau-2p^{\tau-1}+p+4 \le m_0 \le 3 p^{\tau +1}+  p^\tau -4$. Then the equations $f=g=0$ have a non-trivial $p$-adic solution.
        \end{Lemma}
        
        \begin{proof}
		By Lemma \ref{5.1}, $r\ge0$ and $m_0 \ge 2p^{\tau+1} -p^{\tau}$, one can contract the $E^0$ variables to $p^{\tau}$ variables of type $P^1$. 
		
		If there is a $\nu$ such that $I_\nu^1 \ge 2 p^{\tau}+ 4 p^{\tau-1} - \frac{p-1}{2} \sum_{i=1}^{\tau-2}p^i - p^{\tau-2} + 3 \sum_{i=0}^{\tau-3} p^i -2p-2$, one can contract the variables of type $E_\nu^1$ with Lemma \ref{hinterlassen} and the resulting variables together with the variables of type $P^1$ to a variable of type $P^{\tau+1}$ with Lemma \ref{ziel}. From now on, one can assume that 
		\begin{align*}
            I_\nu^1 \le  2 p^{\tau}+4p^{\tau-1} - \frac{p-1}{2} \sum_{i=1}^{\tau-2}p^i - p^{\tau-2} + 3 \sum_{i=0}^{\tau-3} p^i -2p-3
        \end{align*}
        for all $\nu$.
		
        If $m_1 \ge 3p^{\tau} + 5p^{\tau-1} - \frac{p-1}{2} \sum_{i=1}^{\tau-2} p^i -p^{\tau-2} + 3 \sum_{i=0}^{\tau-3} p^i-3p-4$, it follows therefore, that
        \begin{align*}
            q_1=m_1-I_\mx^1 \ge p^{\tau}+ p^{\tau-1} - p-1=\left(p^{\tau-1} + 2\sum_{i=0}^{\tau-2} p^i-1\right)\left(p-1\right)
        \end{align*}
        and
        \begin{align*}
            m_1 \ge2p^{\tau} + 2p^{\tau-1} + p^2 -5p-1= \left(p^{\tau-1} + 2\sum_{i=0}^{\tau-2} p^i-1\right)\left(2p-2\right)+ p^2 -3p+1.
        \end{align*}
        Hence, one can use Lemma \ref{x-y-Lemmaneu} with $x=p^{\tau-1} -p^{\tau-2} -2 \sum_{i=0}^{\tau-3} p^i -1$, $y=p^{\tau-1} + 2 \sum_{i=0}^{\tau-2} p^i-1$ and $z=p-1$ to contract the $E^1$ variables together with the $P^1$ variables to obtain $2p^{\tau-1} + p^{\tau-2}-2$ variables of type $P^2$. Then one can contract them to $2p-1$ variables of type $P^{\tau}$ with Lemma~\ref{fromitojwithout2p} and these to one $P^{\tau+1}$ variable with Lemma \ref{p^n=2}. Thus, one can assume, that
        \begin{align*}
            m_1 \le 3p^{\tau} + 5p^{\tau-1} - \frac{p-1}{2} \sum_{i=1}^{\tau-2} p^i -p^{\tau-2} + 3 \sum_{i=0}^{\tau-3} p^i-3p-5.
        \end{align*}
      
        If one has the even stronger upper bound $m_1 \le 2p^2-p-3$, the $p$-normalisation of $f, g$ can be used to obtain the lower bounds
        \begin{align*}
            m_2 &\ge 6p^{\tau+1}-6p^{\tau}+1 -3 p^{\tau +1}- p^\tau +4 -2p^2+p+3\\
            &=3p^{\tau+1} -7p^{\tau}-2p^2+p+8\ge 2p^{\tau-1} + 2p^{\tau-2}+p^2 -3p-3\\
            &=\left(p^{\tau-2} + 2 \sum_{i=0}^{\tau-3} p^i\right) \left(2p-2\right) + p^2-3p+1
        \end{align*}
        and
        \begin{align*}
            q_2 &\ge 5p^{\tau+1}-5p^{\tau}+1 -3 p^{\tau +1}- p^\tau +4 -2p^2+p+3\\
            &=2p^{\tau+1} -6p^{\tau}-2p^2+p+8\ge p^{\tau-1} + p^{\tau-2}-2\\
            &=\left(p^{\tau-2} + 2 \sum_{i=0}^{\tau-3} p^i\right) \left(p-1\right).
        \end{align*}
        One can contract the $P^1$ variables to $p^{\tau-1}-2$ variables of type $P^2$ using Lemma \ref{fromitojwithout2p}. For $\tau=2$ one can use Lemma \ref{neu} to contract one of the $P^2$ variables together with the $E^2$ variables to a $P^3=P^{\tau+1}$ variable, because $p^{2-2} + 2 \sum_{i=0}^{2-3} p^i =1$.
        For $\tau\ge 3$ on the other hand, one can use Lemma \ref{x-y-Lemmaneu} with $x=p^{\tau-2} - p^{\tau-3} - 2\sum_{i=0}^{\tau-4} p^i -1$, $y = p^{\tau-2} + 2 \sum_{i=0}^{\tau-3} p^i$ and $z=p-4$ to contract the $P^2$ variables to $2p^{\tau-2}+ p^{\tau-3}-1$ variables of type $P^3$. Then one can use Lemma \ref{fromitojwithout2p} to contract them to $2p-1$ variables of type $P^{\tau}$ and Lemma \ref{p^n=2} to obtain a $P^{\tau+1}$ variable. One can therefore assume, that $m_1 \ge 2p^2-p-2=\left(2p-3\right)\left(p+1\right) +1$, from which it follows that there is a $\nu$ such that
        \begin{align*}
            I_\nu^1\ge 2p-2.
        \end{align*}

        One can contract the $p^{\tau}$ variables of type $P^1$ together with the $2p-2$ variables of type $E_\nu^1$ to $p^{\tau-1}$ variables of type $P^2$ with Lemma \ref{fromitoj}. The $p$-normalisation of $f, g$ can be used to obtain the lower bound
        \begin{align*}
	        m_2 \ge 6p^3 -6p^2+1-3p^3-p^2+4-3p^2-2p+6=3p^3-10p^2-2p+11 \ge p^2-p-1
        \end{align*}
        for $\tau=2$ and
        \begin{align*}
            m_2 &\ge 6p^{\tau+1}-6p^{\tau}+1 -3 p^{\tau +1}- p^\tau +4 -3p^{\tau} - 5p^{\tau-1} + \frac{p-1}{2} \sum_{i=1}^{\tau-2} p^i +p^{\tau-2} - 3 \sum_{i=0}^{\tau-3} p^i+3p+5\\
            &=3p^{\tau+1} -10 p^{\tau} -5p^{\tau-1} +\frac{p-1}{2} \sum_{i=1}^{\tau-2} p^i + p^{\tau-2} - 3 \sum_{i=0}^{\tau-3} p^i+3p +10\\
            &\ge 2p^{\tau}+ 6p^{\tau-1} + 3 p^{\tau-2} + 2p^{\tau-3} + 6\sum_{i=0}^{\tau-4} p^i\ge\left(p+1\right) \left( 2 p^{\tau-1} +4 p^{\tau-2}- p^{\tau-3} + 3 \sum_{i=0}^{\tau-4} p^i \right).
        \end{align*}
        for $\tau \ge 3$. Thus, there is a $\mu$ with $I_\mu^2\ge p-1$ for $\tau =2$ and a $\mu$ with
        \begin{align*}
            I_\mu^2 \ge 2 p^{\tau-1} +4 p^{\tau-2} - \frac{p-1}{2} \sum_{i=1}^{\tau-3}p^i - p^{\tau-3} + 3 \sum_{i=0}^{\tau-4} p^i -2p-2,
        \end{align*}
        for $\tau \ge 3$. For $\tau=2$, one can contract the $p-1$ variables of type $E_\mu^2$ together with the $p$ variables of type $P^2$ to a $P^3=P^{\tau+1}$ variable with Lemma \ref{2p-1}. If $\tau \ge 3$, one can obtain a $P^{\tau+1}$ by contracting the $E_\mu^2$ variables with Lemma \ref{hinterlassen} and the resulting ones together with the $P^2$ variables with Lemma \ref{ziel}.
    \end{proof}
    This completes the case $r\left(f, g\right) \ge 0$. The following three lemmata will complete the case $\tau \ge 2$ by showing that for every proper $p$-normalised pair $f, g$ with $\tau\ge 2$ and $r\left(f,g\right)=-1$ the equations $f=g=0$ have a non-trivial $p$-adic solution. Here, it is useful to choose strategies depending on the value of $I_0^0$. As for $r\left(f,g\right)\ge0$, some of the bounds will differ for $p=5$ in order to balance that some inequalities only hold for $p\ge7$.
	\begin{Lemma}
		Let $f, g$ be a proper $p$-normalised pair with $\tau \ge 2$, $r=r\left(f,g\right) =-1$, and $I_0^0 \ge 2 p^{\tau+1} +\frac{11}{2} p^{\tau}  - p^{\tau-1} + 3 \sum_{i=0}^{\tau-2} p^i+ 2p^2 -\frac{11}{2}p- 2 $. Then the equations $f=g=0$ have a non-trivial $p$-adic solution. For $p=5$ even $I_0^0\ge 3 p^{\tau+1} - \sum_{i=0}^{\tau} p^i +2p^2-6p+2$ is sufficient.
	\end{Lemma}
	\begin{proof}
		It is sufficient to contract a $C^{\tau+1}$ variable because $I_0^0 \ge p^{\tau+1} +p^\tau-1$ is given, which implies that the colour $0$ is zero-representing.
       By Lemma \ref{properties}, it follows that 
		\begin{align*}
			I_{00}^0 &\ge \frac{I_{0}^0}{p} \ge  2 p^{\tau} +\frac{11}{2} p^{\tau-1}  - p^{\tau-2} + 3 \sum_{i=0}^{\tau-3} p^i+ 2p -\frac{11}{2}
		\end{align*}
		for $p \ge 5$ and 
		\begin{align*}
			I_{00}^0 \ge 3p^\tau - \sum_{i=0}^{\tau-1} p^i +2p -6,
		\end{align*}
		for $p=5$, which is both bigger than $\left(p-1 \right) \left( p^{\tau-1} +p -2 \right) = p^\tau - p^{\tau-1} + p^2 -3p +2$.
		Furthermore, by \eqref{NichtI_00Grenze}, one obtains
		\begin{align*}
			I_0^0-I_{00}^0 \ge  2p^{\tau+1} -2p^\tau-q_0 - \left( m_1 -I_0^1 \right)\ge p^{\tau+1} -2p^{\tau}+1 -\left(m_1-I_0^1\right),
		\end{align*}
		as $q_0 \le p^{\tau+1}-1$ due to $r=-1$.
		This is bigger than $p^{\tau-1} + p-2 -\left( m_1 -I_0^1 \right)$, therefore, one can take $p^{\tau-1} + p-2 -\left( m_1 -I_0^1 \right)$ sets containing one variable of type $E_{0\bar0}^0$ and  $p-1$ variables of type $E_{00}^0$. By Lemma \ref{createavariableinm_1-I_0^1}, each of this set contains a contraction to a $E_{\bar 0}^1$ variable. For $p \ge 5$ there are at least 
		\begin{align*}
			2 p^{\tau+1} +5 p^{\tau} - \frac{p-1}{2} \sum_{i=1}^{\tau-1}p^i - p^{\tau-1} + 3 \sum_{i=0}^{\tau-2} p^i+ p^2 - 4p-2 
		\end{align*}
		and for $p=5$ at least
		\begin{align*}
			3 p^{\tau+1} -2 p^{\tau} - \sum_{i=0}^{\tau-1} p^i+p^2 -4p+2
        \end{align*}
		variables of type $E_0^0$ left, which is both at least as big as
		\begin{align*}
			p^{\tau}+4p^2-6p+1=p \left( p^{\tau-1} +p-2\right)+3p^2-4p + 1.
		\end{align*}
		As long as there are at least $p\left(3p-3\right)+1=3p^2-3p+1$ variables of type $E_0^0$ left, one has at least $3p-2$ variables of type $E_{0 \mu}^0$ for some $\mu$.
		Therefore, one can use Lemma \ref{FS3} to contract $p^\tau + p^2-2p$ of the $E_0^0$ variables to $p^{\tau-1} + p-2 $ variables of type $E_0^1$, using each time $p$ variables of the same colour nuance. Now, one has $p^{\tau-1} +p-2$ variables of type $E_0^1$ and $p^{\tau-1}+p-2$ variables of type $E_{\bar 0}^1$. This leaves at least
		\begin{align*}
		2 p^{\tau+1} +4 p^{\tau} - \frac{p-1}{2} \sum_{i=1}^{\tau-1}p^i - p^{\tau-1} + 3 \sum_{i=0}^{\tau-2} p^i -2p- 2
		\end{align*}
		 variables of type $E_0^0$ for $p\ge 5$ and 
		 \begin{align*}
		 3 p^{\tau+1} -3 p^{\tau} - \sum_{i=0}^{\tau-1} p^i -2p+2
		 \end{align*}
		  for $p=5$ remaining. 
		Use Lemma~\ref{hinterlassen} for $p\ge5$ and Lemma \ref{hinterlassen5} for $p=5$ to contract the $E_0^0$ variables to $p-1$ variables of type $E_0^\tau$ and $2p-2$ variables of type $E_0^i$ for all $i \in \{0,1, \dots, \tau-1\}$. With Lemma~\ref{fromitoj}, one can contract $p^{\tau+1}-p^{\tau}$ of the variables of type $C^0$ and the $2p-2$ variables of type $E_0^0$ to $p^{\tau}-p^{\tau-1}$ variables of type $C^1$.
		Use Lemma \ref{x-y-Lemma} with $x= p^{\tau-1} - \sum_{i = 0}^{\tau-2} p^i -1$, $y = \sum_{i=0}^{\tau-2} p^i +1$ and $z= p-2$ to contract $p^{\tau-1}+p-2$ variables of type $E_0^1$ and $p^{\tau-1}+p-2$ variables of type $E_{\bar 0}^1$ together with the $C^1$ variables to $p^{\tau-1}-1$ variables of type $C^{2}$, without using $2p-2\ge p$ of the $C^1$ variables. The $2p-2 \ge p-1$ variables of type $E_0^1$ which where contracted while the $p-1$ variables of type $E_0^\tau$ were contracted are also unused. One can contract $p-1$ of them together with $p$ of the remaining $C^1$ variables to an additional $C^{2}$ variable using Lemma~\ref{2p-1}. This gives a total of $p^{\tau-1}$ variables of type $C^{2}$. Then one can contract the $C^{2}$ variables with the $E_0^i$ variables for $i \in\{2, \dots, \tau-1\}$ and the $E_0^\tau$ variables to a $C^{\tau+1}$ variable due to Lemma \ref{ziel}.
	\end{proof}
	\begin{Lemma}
		Let $f, g$ be a proper $p$-normalised pair with $\tau \ge 2$, $r=r\left(f,g\right)=-1$ and $p^{\tau +1}+ p^\tau -1 \le I_0^0 \le  2 p^{\tau+1} +\frac{11}{2} p^{\tau}  - p^{\tau-1} + 3 \sum_{i=0}^{\tau-2} p^i+2p^2 -\frac{11}{2} p-3$ for $p\ge 7$ and $p^{\tau+1}+p^\tau-1\le I_0^0 \le 3 p^{\tau+1} - \sum_{i=0}^{\tau} p^i +2p^2-6p+1$ for $p=5$. Then the equations $f=g=0$ have a non-trivial $p$-adic solution.
	\end{Lemma}
	\begin{proof}
        It follows from $r=-1$ and the restrictions on $I_0^0$ that
        \begin{align}
        \label{XX}
			m_0 \le 3 p^{\tau+1} +\frac{11}{2} p^{\tau}  -p^{\tau-1} + 3 \sum_{i=0}^{\tau-2} p^i+2p^2 -\frac{11}{2} p-4,
		\end{align}
		for $p\ge 7$, whereas one can obtain for $p=5$ the even better bound
		\begin{align}
		\label{IX}
			m_0 \le 4 p^{\tau+1} - \sum_{i=0}^{\tau} p^i +2p^2-6p.
		\end{align}
		As $I_0^0 \ge  p^{\tau+1} + p^{\tau} -1$ the colour $0$ is zero-representing, hence, it suffices to show that one can contract a $C^{\tau+1}$ variable. Due to the $p$-normalisation of $f, g$ and $r=-1$, one has the lower bound
        \begin{align*}
			I_0^0 \ge 3 p^{\tau+1} - 3 p^{\tau} +1 - q_0 -q_1 \ge 2 p^{\tau+1} - 3 p^\tau +2 -q_1
		\end{align*}
		as well.
		
		Assume first, that $I_\nu^1=I_\mx^1 \ge p^\tau+ 4 p^{\tau-1} - \frac{p-1}{2} \sum_{i=1}^{\tau-2} p^i+ 3 \sum_{i=0}^{\tau-3} p^i -p -4$. Then one can make sure, that additionally, one has $p^{\tau}+p-2$ variables of type $E_{\bar \nu}^1$ by contracting the $E_0^0$ variables to at least $p^{\tau}+p-2-q_1$ variables of type $E_{\bar \nu}^1$ as described in the following paragraph.
		
		One can assume that $q_1 \le p^\tau+p-3$, because else, there is nothing to be done.
		If $\nu\neq 0$, one can contract the variables of type $E_0^0$ to
		\begin{align*}
		\left \lceil \frac{ 2 p^{\tau+1} - 3 p^\tau +2 -q_1}{p} \right \rceil -2p + \frac{p-3}{2} \ge 2p^\tau - 3 p^{\tau-1} + \frac{2-q_1}{p} -2p+ \frac{p-3}{2} 
		\end{align*}
		variables of type $E_0^1$ with Lemma \ref{FS6} for $p\ge 7$, which is at least as big as $p^{\tau}+p-2-q_1$ for $p \ge 7$ and to contract
		\begin{align*}
			\left \lceil \frac{ 2 p^{\tau+1} - 3 p^\tau +2 -q_1}{p} \right \rceil -2p + 3 \ge 2p^\tau - 3 p^{\tau-1} + \frac{2-q_1}{p} -2p+3 \ge p^\tau + p-2 -q_1
		\end{align*}
		variables of type $E_0^1$ with Lemma \ref{FS6} for $p=5$. This leaves $6p-9\ge 2p-2$ variables of type $E_0^0$ unused in both cases.
		If on the other hand, one has $\nu=0$, it follows that
		\begin{align*}
		I_{00}^0 \ge \frac{I_0^0}{p} \ge 2p^\tau -3p^{\tau-1} + \frac{2-q_1}{p}\ge p^\tau+p-2-q_1
		\end{align*}
		and by $m_1-I_0^1=q_1$ and \eqref{NichtI_00Grenze} that
		\begin{align*}
		I_0^0 - I_{00}^0 \ge  2p^{\tau+1}-2p^\tau -q_0 -q_1 \ge p^{\tau+1} - 2p^\tau +1-q_1\ge p^{\tau} +p-2-q_1.
		\end{align*}
		Furthermore, one has
		\begin{align*}
            I_0^0 \ge 2p^{\tau+1} -3p^{\tau}+2-q_1 \ge p^{\tau+1} + 2p^2 -5p-q_1 p +3=p \left(p^{\tau} +p -2 -q_1\right) + p^2-3p+3.
		\end{align*}
		Thus one can contract $p^{\tau+1}+p^2-2p-q_1p$ of the $E_0^0$ variables to $p^{\tau}+ p -2-q_1$ variables of type $E_{\bar 0}^1$ due to Lemma \ref{moreinm_1-I_0^0}, leaving at least $p^{\tau+1} -3p^{\tau}-p^2+2p+2+\left(p-1\right)q_1\ge 2p-2$ variables of type $E_0^0$ unused.
		
		In both cases, one has contracted enough $E_{\bar \nu}^1$ variables to have at least $p^{\tau}+p-2$ variables of type $E_{\bar \nu}^1$, while there are $2p-2$ variables of type $E_0^0$ remaining. The $E_0^0$ variables can be contracted together with $p^{\tau+1}-p^\tau$ of the $C^0$ variables to $ p^{\tau} - p^{\tau-1}$ variables of type $C^1$, using Lemma \ref{fromitoj}. 
		Then, one can contract $4 p^{\tau-1} - \frac{p-1}{2} \sum_{i=1}^{\tau-2} p^i+ 3 \sum_{i=0}^{\tau-3} p^i -2p -2$ of the variables of type $E_\nu^1$ with Lemma \ref{2ptotau-1} to $2p-2$ variables of type $E_\nu^j$ for all $j \in \{1, \dots, \tau-1\}$. The remaining $p^\tau+p-2$ variables of type $E_\nu^1$ together with the $p^{\tau}+p-2$ variables of type $E_{\bar \nu}^1$ and the $C^1$ variables can be contract, using Lemma \ref{x-y-Lemma} with $x=p^{\tau-1} -2p^{\tau-2} - \sum_{i=0}^{\tau-3} p^i -1$, $y= \sum_{i=0}^{\tau-1} p^i+1$ and $z=p-2$, to $2p^{\tau-1}-p^{\tau-2}$ variables of type $C^{2}$.
		With Lemma \ref{ziel} those and the $2p-2$ variables in $E_\nu^j$ for $j \in \{2, \dots, \tau-1\}$ can be contracted to a $C^{\tau+1}$ variable.
        Thus, from now on, one can assume, that
        \begin{align}
        \label{X}
            I_\mx^1 \le p^\tau+ 4 p^{\tau-1} - \frac{p-1}{2} \sum_{i=1}^{\tau-2} p^i+ 3 \sum_{i=0}^{\tau-3} p^i-p -5.
        \end{align}
		If $q_1 \ge p^\tau + p -2= \left(\sum_{i=0}^{\tau-1} p^i +1\right) \left( p-1\right)$ and $m_1 \ge 2p^\tau +p^2-p -3= \left(\sum_{i=0}^{\tau-1} p^i +1\right) \left( 2p-2\right)+ p^2 -3p +1$, one can use Lemma \ref{2ptotau-1} to contract the $E_0^0$ variables to $2p-2$ variables of type $E_0^i$ for all $i \in \{ 0, \dots, \tau-1\}$ because $I_0^0 \ge p^{\tau+1} + p^\tau-1 \ge 4 p^{\tau} - \frac{p-1}{2} \sum_{i=1}^{\tau-1} p^i+ 3 \sum_{i=0}^{\tau-2} p^i -2p -2 $. By Lemma~\ref{fromitoj}, the $p^{\tau+1}-p^\tau$ variables of type $C^0$ can be contracted together with the $2p-2$ variables of type $E_0^0$ to $p^{\tau}-p^{\tau-1}$ variables of type $C^1$. Using Lemma \ref{x-y-Lemmaneu} with $x=p^{\tau-1} -2p^{\tau-2} -\sum_{i=0}^{\tau-3} p^i-1$, $y=  \sum_{i=0}^{\tau-1} p^i +1$ and $z=p-2$, one can contract the $E^1$ variables together with the $C^1$ variables to $2p^{\tau-1} - p^{\tau-2}$ variables of type $C^2$, which contract together with the $2p-2$ variables of type $E_0^i$ for $i \in \{2, \dots, \tau-1\}$ to a $C^{\tau+1}$ variables due to Lemma \ref{ziel}.
        Therefore, one can assume that either $m_1 \le 2p^\tau +p^2-p -4$ or $q_1 \le p^\tau + p -3$. The latter case leads to $m_1 = q_1+I_\mx^1 \le 2p^\tau+ 4 p^{\tau-1} - \frac{p-1}{2} \sum_{i=1}^{\tau-2} p^i+ 3 \sum_{i=0}^{\tau-3} p^i -8$ due to \eqref{X}. Hence, from now on, one can assume that
        \begin{align}
        \label{XXX}
            m_1 \le 2p^\tau+ 4 p^{\tau-1} - \frac{p-1}{2} \sum_{i=1}^{\tau-2} p^i+ 3 \sum_{i=0}^{\tau-3} p^i+p^2  -8,
        \end{align}
        because this is an upper bound for the upper bound for $m_1$ in both cases.
        
        By the $p$-normalisation of $f, g$, it follows that
        \begin{align}
        \label{I_0}
            I_0^0 \ge 4p^{\tau+1}-4p^{\tau} +1 -q_0 -m_1 \ge 3p^{\tau+1}-6p^{\tau}- 4 p^{\tau-1} - 3 \sum_{i=0}^{\tau-3} p^i-p^2+10.
        \end{align}
        Therefore, one has $I_{00}^0 \ge 3p^{\tau}-6p^{\tau-1}-4p^{\tau-2}-3\sum_{i=0}^{\tau-4} p^i -p  \ge p^{\tau-1}+2p-3$ and, due to \eqref{NichtI_00Grenze}, it follows that
        \begin{align*}
            I_0^0-I_{00}^0 &\ge 2p^{\tau+1} -2p^{\tau} -q_0 -\left(m_1-I_0^1\right) \ge p^{\tau+1} -2p^{\tau} +1 -\left(m_1-I_0^1\right) \\ 
            &\ge p^{\tau-1} + 2p-3 -\left(m_1-I_0^1\right).
        \end{align*}
        It follows from \eqref{I_0} that $I_0^0 \ge p^{\tau} + 2p^2-3p -p\left(m_1-I_0^1\right)+ p^2-3p+3$, thus, if $m_1-I_0^1 \le p^{\tau-1}+2p-3$, one can contract $p^{\tau} + 2p^2-3p -p\left(m_1-I_0^1\right)$ of the $E_0^0$ variables to $p^{\tau-1} + 2p-3-\left(m_1-I_0^1\right)$ variables of type $E_{\bar 0}^1$ with Lemma \ref{moreinm_1-I_0^0}. There are at least $3p^{\tau+1}-7p^{\tau}- 4 p^{\tau-1} - 3 \sum_{i=0}^{\tau-3} p^i-3p^2+3p+10$ variables of type $E_0^0$ remaining, which contract to
        \begin{align*}
            \left \lceil  \frac{3p^{\tau+1}-7p^{\tau}- 4 p^{\tau-1} - 3 \sum_{i=0}^{\tau-3} p^i-3p^2+3p+10}{p} \right \rceil -2p+ \frac{p-3}{2} \\
            \ge 3p^{\tau}-7p^{\tau-1}- 4 p^{\tau-2} - 3 \sum_{i=0}^{\tau-4} p^i-5p+\frac{p-3}{2}+4
        \end{align*}
        variables of type $E_0^1$ with Lemma \ref{FS6}, while leaving at least $6p-9\ge2p-2$ variables of type $E_0^0$ unused. This is at least as big as $p^{\tau-1}+2p-3$. Thus, one has at least $p^{\tau-1} +2p-3$ variables of type $E_0^1$, as well as a total of $p^{\tau-1} +2p-3$ variables of type $E_{\bar 0}^1$.
        
       By Lemma \ref{fromitoj}, one can contract $p^{\tau+1} -p^{\tau}$ of the $C^0$ variables with the remaining $2p-2$ variables of type $E_0^0$ to $p^{\tau}-p^{\tau-1}$ variables of type $C^1$ and then use Lemma \ref{x-y-Lemma} with $x=p^{\tau-1} -\sum_{i=0}^{\tau-2} p^i -1$, $y=\sum_{i=0}^{\tau-2} p^i+2$ and $z=p-3$ to contract them together with the $E^1$ variables to $p^{\tau-1}$ variables of type $C^2$. 
        
        For $\tau=2$ it follows for $p \ge 7$ due to \eqref{XX} and \eqref{XXX}, that
        \begin{align*}
            m_2 \ge 3 p^{3} -\frac{33}{2} p^{2}  +\frac{5}{2} p+10\ge p^2-p-1=\left(p-2\right)\left(p+1\right) +1,
        \end{align*}
        and for $p=5$ due to \eqref{IX} and \eqref{XXX}, that
        \begin{align*}
            m_2 \ge 2p^{3} -10 p^{2}  +3 p+10\ge p^2-p-1=\left(p-2\right)\left(p+1\right) +1.
        \end{align*}
        Therefore, one has a $\mu$ with $I_\mu^2 \ge p-1$, from which it follows that one can contract the $p$ variables of type $C^2$ and the $p-1$ variables of type $E_\mu^2$ to a $C^3=C^{\tau+1}$ variable due to Lemma~\ref{2p-1}. Thus, from now on one can assume, that $\tau \ge 3$.
        
        If $I_\mu^2=I_\mx^2 \ge2 p^{\tau-1} +4 p^{\tau-2} - \frac{p-1}{2} \sum_{i=1}^{\tau-3}p^i -p^{\tau-3} + 3 \sum_{i=0}^{\tau-4} p^i -2p-2$, one can use Lemma \ref{hinterlassen} to contract the $E_\mu^2$ variables to $p-1$ variables of type $E_\mu^{\tau}$ and $2p-2$ variables of type $E_{\mu}^i$ for all $i \in\{2, \dots, \tau-1\}$. It follows that one can contract them together with the $C^2$ variables to a $C^{\tau+1}$ variable due to Lemma \ref{ziel}. From now on, one can assume, that
        \begin{align*}
            I_\mx^2 \le2 p^{\tau-1} +4 p^{\tau-2} - \frac{p-1}{2} \sum_{i=1}^{\tau-3}p^i -p^{\tau-3} + 3 \sum_{i=0}^{\tau-4} p^i -2p-3,
        \end{align*}
        and therefore
        \begin{align}
        \label{III}
            m_2 \le 2p^{\tau} +6 p^{\tau-1} +3 p^{\tau-2} +2p^{\tau-3} + 6 \sum_{i=0}^{\tau-4} p^i -2p^2-5p-3.
        \end{align}
        Then, one can contract the $p^{\tau-1}$ variables of type $C^2$ to $p^{\tau-2}-2$ variables of type $C^3$ using Lemma \ref{fromitojwithout2p}.
        Due to \eqref{XX}, \eqref{XXX} and \eqref{III}, it follows that
        \begin{align*}
            m_0 + m_1+ m_2 \le 3 p^{\tau+1} +\frac{19}{2} p^{\tau}  +\frac{17}{2}p^{\tau-1} + 6p^{\tau-2} +8 p^{\tau-3} +12 \sum_{i=0}^{\tau-4} p^i+p^2 -10p-15, 
        \end{align*}
        which does not only hold for $p \ge7$ but also for $p=5$ because the upper bound \eqref{XX} is in the case $p=5$ bigger than the upper bound \eqref{IX}. This leads to
        \begin{align*}
            q_3 \ge 4p^{\tau+1} -\frac{33}{2} p^{\tau} -\frac{17}{2}p^{\tau-1} - 6p^{\tau-2}-8p^{\tau-3} -12 \sum_{i=0}^{\tau-4} p^i-p^2 +10 p+16\ge p^{\tau-2}+p^{\tau-3}-2
        \end{align*}
        and 
        \begin{align*}
            m_3 &\ge 5p^{\tau+1} -\frac{35}{2} p^{\tau} -\frac{17}{2}p^{\tau-1} - 6p^{\tau-2}-8p^{\tau-3} -12 \sum_{i=0}^{\tau-4} p^i-p^2 +10 p+16\\
            &\ge 2p^{\tau-2}+2p^{\tau-3}+p^2 -3p-3. 
        \end{align*}
        For $\tau=3$ one can contract one of the $C^3$ variables together with the $E^3$ variables to a $C^4$ variable using Lemma \ref{neu} with $x=1$. For $\tau\ge4$ the $C^3$ variables can be contracted with the $E^3$ variables, using Lemma \ref{x-y-Lemmaneu} with $x=p^{\tau-3}-p^{\tau-4}-2\sum_{i=0}^{\tau-5} p^i-1$, $y=p^{\tau-3} + 2 \sum_{i=0}^{\tau-4} p^i$ and $z=p-4$, to $2p^{\tau-3}+p^{\tau-4}-1$ variables of type $C^4$. Then one can use Lemma \ref{fromitojwithout2p} to contract them to $2p-1$ variables of type $C^{\tau}$ and then Lemma \ref{p^n=2} to contract them to a $C^{\tau+1}$ variable.
	\end{proof}
	\begin{Lemma}
		Let $f, g$ be a proper $p$-normalised pair with $\tau \ge 2$, $r=r\left(f,g\right)=-1$ and $I_0^0 \le p^{\tau+1} + p^\tau -2$. Then the equations $f=g=0$ have a non-trivial $p$-adic solution.
	\end{Lemma}
	\begin{proof}
		Due to the upper bound for $I_0^0$ and $r=-1$ it follows that
		\begin{align}
			\label{o}
			m_0 \le p^{\tau+1} + p^\tau -2 + p^{\tau+1}-1 = 2 p^{\tau+1} + p^\tau -3
		\end{align}
		and hence,
		\begin{align}
		\label{i}
			q_1 \ge 3 p^{\tau+1}-3 p^{\tau} +1 -2 p^{\tau+1} - p^\tau +3 = p^{\tau+1} -4 p^{\tau} +4 \ge p^{\tau}+p-2
		\end{align}
		and	
		\begin{align}
		\label{ii}
			m_1\ge 4 p^{\tau+1}-4p^{\tau}+1 -2 p^{\tau+1} - p^\tau +3 = 2 p^{\tau+1} -5 p^\tau +4 \ge 2p^{\tau}+p^2-p-3.
		\end{align}
		Use Lemma \ref{5.1} to contract the $E^0$ variables to $p^{\tau}-p^{\tau-1}$ variables of type $P^1$. 
		
		If $I_\nu^1=I_\mx^1 \ge p^\tau+ 4 p^{\tau-1} - \frac{p-1}{2} \sum_{j=1}^{\tau-2} p^j+ 3 \sum_{j=0}^{\tau-3} p^j -p -4$, one can contract $4 p^{\tau-1} - \frac{p-1}{2} \sum_{i=1}^{\tau-2} p^i+ 3 \sum_{i=0}^{\tau-3} p^i -2p -2$ of the variables of type $E_\nu^1$ to $2p-2$ variables of type $E_\nu^j$ for all $j \in \{1, \dots, \tau-1\}$ using Lemma \ref{2ptotau-1}, which leaves $p^\tau+p-2$ variables of type $E_\nu^1$ unused.
        Then, Lemma \ref{x-y-Lemma} can be used with $x=p^{\tau-1} -2p^{\tau-2} - \sum_{i=0}^{\tau-3} p^i -1$, $y= \sum_{i=0}^{\tau-1} p^i+1$ and $z=p-2$ to contract the remaining $E_\nu^1$ variables together with the $p^{\tau}+p-2$ variables of type $E_{\bar \nu}^1$ and the $P^1$ variables to $2p^{\tau-1}-p^{\tau-2}$ variables of type $P^{2}$.
		Those and the $2p-2$ variables in $E_\nu^j$ for $j \in \{2, \dots, \tau-1\}$ can be contracted to a $P^{\tau+1}$ variable, using Lemma \ref{ziel}.
		
		Thus, one can furthermore assume that one has $ I_\nu^1 =I_\mx^1\le p^\tau+ 4 p^{\tau-1} - \frac{p-1}{2} \sum_{i=1}^{\tau-2} p^i+ 3 \sum_{i=0}^{\tau-3} p^i -p -5 $. It follows, that
		\begin{align}
		\label{oo}
			m_1 \le p^{\tau+1} + 5 p^{\tau} + 4 p^{\tau-1} +3 p^{\tau-2} + 6 \sum_{i=0}^{\tau-3} p^i.
		\end{align}
        Due to \eqref{i} and \eqref{ii}, one can use Lemma \ref{x-y-Lemmaneu} with $x= p^{\tau-1} -2p^{\tau-2} - \sum_{i=0}^{\tau-3}p^i-1$, $y=\sum_{i=0}^{\tau-1} p^i+1$ and $z=p-2$ to contract the $p^{\tau}-p^{\tau-1}$ variables of type $P^1$ and the $E^1$ variables to $2p^{\tau-1}-p^{\tau-2}$ variables of type $P^2$. For $\tau=2$, one can use Lemma \ref{p^n=2} to contract the $2p-1$ variables of type $P^2$ to a $P^3=P^{\tau+1}$ variable. Hence, one can assume that $\tau \ge 3$.
        As a consequence of \eqref{o} and \eqref{oo}, it follows that
        \begin{align*}
            m_2 &\ge 6p^{\tau+1} -6p^{\tau}+1 -2 p^{\tau+1} -p^\tau +3-p^{\tau+1} - 5 p^{\tau} - 4 p^{\tau-1} -3 p^{\tau-2} - 6 \sum_{i=0}^{\tau-3} p^i\\
            &=3p^{\tau+1} -12 p^{\tau} -4 p^{\tau-1} -3 p^{\tau-2} -6\sum_{i=0}^{\tau-3} p^i +4,
        \end{align*}
        which is bigger than $\left(p+1\right) \left(4 p^{\tau-2} - \frac{p-1}{2} \sum_{i=1}^{\tau-3} p^i+ 3 \sum_{i=0}^{\tau-4} p^i -2p -2\right)$. Hence, there is a $\mu$ such that $I_\mu^2 \ge 4 p^{\tau-2} - \frac{p-1}{2} \sum_{i=1}^{\tau-3} p^i+ 3 \sum_{i=0}^{\tau-4} p^i -2p -2$, thus, one can contract the $E_\mu^2$ variables using Lemma \ref{2ptotau-1} and then the resulting variables together with the $P^2$ variables to a $P^{\tau+1}$ variable, using Lemma \ref{ziel}.
    \end{proof}
	   It follows that for a proper $p$-normalised pair $f, g$ with $\tau \ge 2$ the equations $f=g=0$ have a non-trivial $p$-adic solution, which in combination with Section \ref{Pairs of forms with tau=1} proves the claim of the theorem.
    
    \bibliographystyle{plain}
	\bibliography{literatur}

\end{document}